\documentclass[reqno,11pt]{amsart}

\usepackage[T1]{fontenc}
\usepackage[utf8]{inputenc} 
\usepackage{lmodern}

\usepackage{amsmath,amssymb,amsthm}
\usepackage{mathtools}
\usepackage{bbm}     
\usepackage{dsfont}  

\usepackage{graphicx}
\usepackage{booktabs}
\usepackage{multirow}
\usepackage{placeins}
\usepackage{lscape}
\usepackage{subcaption} 

\usepackage{xcolor}
\usepackage{ulem}
\usepackage[a4paper,margin=3.0cm]{geometry}

\usepackage[nesting]{hyperref}
\hypersetup{
  colorlinks=true,
  linkcolor=blue,
  citecolor=blue,
  urlcolor=blue
}

\numberwithin{equation}{section}
\numberwithin{figure}{section}

\theoremstyle{plain}
\newtheorem{thm}{Theorem}[section]

\newtheorem{lem}[thm]{Lemma}
\newtheorem{cor}[thm]{Corollary}

\theoremstyle{definition}
\newtheorem{defn}[thm]{Definition}
\newtheorem{ex}[thm]{Example}

\theoremstyle{remark}
\newtheorem{rmk}[thm]{Remark}

\newcommand{\R}{\mathbb{R}}
\newcommand{\C}{\mathbb{C}}

\newcommand{\N}{\mathbb{N}}

\makeatletter
\AtBeginDocument{%
 \def\@serieslogo{%
 \vbox to\headheight{%
 \parindent\z@ \fontsize{6}{7\p@}\selectfont
 \vss}}}

\def\makeoverbar#1#2#3#4#5#6#7{%
 \setbox0=\hbox{$\m@th#2\mkern#5mu{{}#3{}}\mkern#6mu$}%
 \setbox1=\null \dimen@=#4\fontdimen8#13 \dimen@=3.5\dimen@
 \advance\dimen@ by \ht0 \dimen@=-#7\dimen@ \advance\dimen@ by \wd0
 \ht1=\ht0 \dp1=\dp0 \wd1=\dimen@
 \dimen@=\fontdimen8#13 \fontdimen8#13=#4\fontdimen8#13
 \rlap{\hbox to \wd0{$\m@th\hss#2{\overline{\box1}}\mkern#5mu$}}
 \fontdimen8#13=\dimen@}

\def\mylabel#1#2{{\def\@currentlabel{#2}\label{#1}}}

\makeatother

\begin{document}

\title[The Characterization of $\mathcal{D}^{\alpha,2}$]{Characterization of the (fractional) Malliavin--Watanabe--Sobolev spaces $\mathcal{D}^{\alpha,2}$ via the Bargmann--Segal norm}

\author{Wolfgang Bock${}^1$}
\address{Linnaeus University, Department of Mathematics, Universitetsplatsen 1, 352 52 V\"axj\"o, Sweden}

\email{\href{mailto:wolfgang.bock@lnu.se}
{wolfgang.bock@lnu.se}}

\author{Martin Grothaus${}^2$}
\address{RPTU University Kaiserslautern--Landau, Department of Mathematics, Kaiserslautern Campus, Germany}

\email{\href{mailto:grothaus@mathematik.uni-kl.de} {grothaus@rptu.de}}

\begin{abstract}
Motivated by an open question going back to P.~Malliavin and P.-A.~Meyer (and closely related to the foundational work of S.~Watanabe) on whether Malliavin--Watanabe--Sobolev regularity admits a characterization in terms of a holomorphic La\-place image similar as for Hida distributions, we establish a characterization of the spaces $\mathcal{D}^{\alpha,2}$ for all $\alpha\in\mathbb{R}$ via the Bargmann--Segal norm of the $S$-transform. More precisely, we express $\mathcal{D}^{\alpha,2}$-regularity, $\alpha > 0$, of $F\in L^{2}(\mu)$, as well as dual regularity of distributions, in terms of integrability, differentiability and growth properties of the function
\[
(0,1) \ni \lambda \longmapsto \int_{\mathcal{S}'_{\mathbb{C}}} |SF(\lambda u)|^{2}\,d\nu(u)
\]
involving integer-order derivatives in $\lambda$ for $\alpha\in\mathbb{N}$ and Riemann-Liouville fractional derivatives/integrals for non-integer $\alpha$. Here $\nu$ is the Gaussian Bargmann--Segal measure. This yields practical criteria for both positive and negative (including fractional) orders of Malliavin regularity and thereby bridges Malliavin calculus and Bargmann--Segal techniques from white noise analysis. Applications are worked out for Donsker's delta, self-intersection local times of Gaussian processes, and Gauss kernels.
\end{abstract}
\maketitle

\vspace{0.5cm}
\begin{minipage}{14cm}
{\small
\begin{description}
\item[\rm \textsc{Keywords}] Malliavin differentiability, characterization theorem, regularity
\item[\rm \textsc{Mathematics Subject Classification}] 60H07, 60H40, 46E35, 46C15
\end{description}
}
\end{minipage}

\section{Introduction}
 In Gaussian analysis concepts for test-, generalized-function, and Sobolev spaces on infinite dimensional spaces are developed. Due to the absence of a Lebesgue measure, Gaussian measures are the natural reference measures. Very successful special cases are Malliavin calculus, see, e.g. \cite{Bell, Nualart} and white noise analysis, see, e.g. \cite{HKPS93, Ob94}. We also refer to \cite{Yan} for an overview of how both are connected. Their methods and tools have become established in fields such as stochastic analysis, e.g \cite{BP98, VenturiEtAl13}, finance, e.g. \cite{AOPU00, AO18} and mathematical physics, e.g.\cite{LetUsUse}. While Malliavin calculus is famous due to its Malliavin--Watanabe--Sobolev spaces used in the calculus of variations see e.g.\cite{Nualart, NR97}, white noise theory gives access to stochastic distributions and generalized solutions of SPDEs see e.g. \cite{KLPStW96, GNM22}. Both calculi have in common that the Wiener--It\^{o} chaos decomposition plays a central role. 

The main advantage of white noise theory is its profound analysis, which is based on various characterization theorems and their corollaries, see \cite{PS91, KLPStW96, GNM22}. These allow the users to reduce the study of complex infinite-dimensional objects and structures by suitable spaces of holomorphic functions after having applied a Laplace or Fourier transform. 
Paul Malliavin and Paul--Andre Meyer already saw the potential of this powerful toolbox and gave the pioneers of the first characterization theorems, Hida, Potthoff and Streit, the homework to find a similar characterization for the Malliavin-Watanabe-Sobolev spaces.

This homework initiated the characterization and study of several new Gel'fand triples refining the existing triple of Hida test functions and distributions. Among these triples there is the well-studied Potthoff-Timpel triple \cite{PT94}. It was characterized using the Bargmann-Segal space in \cite{GKS97} and revisited recently, by proving a much easier to apply characterization which involves the Bargmann-Segal norm of the Gauss--Laplace transform, see \cite{GNM22}.

However, these triples are having too small test-function spaces. Already Potthoff and Timpel showed, using Nelsons hypercontractivity theorem, that the Potthoff--Timpel test-functions are a subset of the infinitely often Malliavin differentiable functions with Malliavin derivative in all $L^p(\mu), 1\le p < \infty$. That is of course bad news for the characterization of the Malliavin--Watanabe--Sobolev spaces and here new ideas are needed.

In this article we give a complete characterization of the Malliavin--Watanabe--Sobolev spaces $\mathcal{D}^{\alpha,2}$ for all $\alpha \in \mathbb{R}$ using the Bargmann--Segal norm.

We demonstrate its simple and practical applicability of our characterization in several examples. Here, we would like to stress that it can be applied in two situations. First, for proving smoothness of square-integrable functions, second for deriving regularity of distributions. Since we can consider $\mathcal{D}^{\alpha,2}$ even for fractional $\alpha \in \mathbb{R}$, questions about optimal regularity can also be answered very precisely.

Note that with this we close a gap in the theory open for more than 25 years, bridging between Malliavin calculus and white noise analysis.

\section{Preliminaries from White Noise Analysis}
The starting point for white noise analysis is the Gel'fand triple
\begin{equation*}
	\mathcal{S} \subset L^{2}(\mathbb{R}) \subset \mathcal{S}^{\prime}. 
\end{equation*}

The inner product and norm of $L^{2}(\mathbb{R}):=L^{2}(\mathbb{R}, dx; \mathbb{R})$, the Hilbert space of real valued square integrable functions w.r.t.~the Lebesgue measure on $\mathbb{R}$, are denoted by $(\cdot,\cdot)$ and $|\cdot|$, respectively. The symbol $\mathcal{S}^{\prime}:=\mathcal{S}^{\prime}(\mathbb{R}; \mathbb{R})$ stands for the space of tempered distributions, the collection of all continuous linear functionals mapping the Schwartz space of real valued, smooth, rapidly decreasing functions $\mathcal{S}:=\mathcal{S}(\mathbb{R}; \mathbb{R})$ to $\mathbb{R}$. The dual pairing between {$x \in \mathcal{S}^{\prime}$} and $\xi \in \mathcal{S}$ is denoted by $\langle \xi, x \rangle$, which is an extension of the inner product in $L^{2}(\mathbb{R})$, i.e.,
\begin{equation*}
	\langle \xi, f \rangle = (\xi,f), \quad f \in L^{2}(\mathbb{R}), \, \xi \in \mathcal{S}.
\end{equation*}
The spaces $\mathcal{S}$ and $\mathcal{S}^{\prime}$ are endowed with their usual projective limit and inductive limit topology, respectively, corresponding to a countable system of semi-norms. The $\sigma$-algebra generated by the inductive limit topology on $\mathcal{S}^{\prime}$
coincides with the $\sigma$-algebra generated by the cylinder sets
$$\mathcal{C}_{\sigma} := \{x \in \mathcal{S}^{\prime}  :  (\langle \xi_1,x\rangle,\ldots,\langle \xi_n,x\rangle )\in B,\,\xi_1, \ldots \xi_n \in \mathcal{S}, B \in  \mathcal{B}(\mathbb{R}^{n}), n \in \N  \}, \nonumber$$
where $ \mathcal{B}(\mathbb{R}^{n})$, for $n \in \N$, is the Borel $\sigma$-algebra on $\mathbb{R}^{n}$. Applying the Bochner--Minlos theorem we introduce the standard Gaussian measure $\mu$ on $(\mathcal{S}^{\prime},\mathcal{C}_{\sigma})$ given by its characteristic function

\begin{equation*}
	\mathcal{S} \ni\xi \longmapsto \int_{\mathcal{S}^{\prime}}\exp{(i\langle \xi, x \rangle)} d\mu(x)=\exp \Big(-\frac{1}{2}|\xi|^{2} \Big).
\end{equation*}

The probability space $(\mathcal{S}^{\prime},\mathcal{C}_{\sigma},\mu)$ is then called the white noise space. The Hilbert space of complex valued square integrable functions with respect to the white noise measure is denoted by 
\begin{equation*}
	{L^{2}(\mu):=L^{2}(\mathcal{S}^{\prime}, \mathcal{C}_{\sigma},\mu;\mathbb{C})}
\end{equation*}
with inner product 
\begin{equation*}
	(\!(F,G)\!) := \int_{\mathcal{S}^{\prime}} F(x)\overline{G}(x)d\mu(x), \text{ for any } F, G \in L^{2}(\mu),
\end{equation*}
and induced norm denoted by $\| \cdot\|$. The subscript $\cdot_{\mathbb{C}}$ denotes complexification. I.e., if $V$ is a real vector space, then $V_{\mathbb{C}}$ denotes its complexification. If $V({\mathbb R}^n)$ is a vector space of (generalized) functions on ${\mathbb R}^n$, $n \in{\mathbb N}$, then $\widehat{V({\mathbb R}^n)}$ denotes its subspace of functions which are symmetric in the $n$ variables.

Each $F \in L^{2}(\mu)$ has the Wiener-It\^{o} expansion or chaos decomposition 
$$F = \sum_{n=0}^\infty \big\langle F^{(n)}, :\cdot^{\otimes n}: \big\rangle, \quad F^{(n)} \in \widehat{L^2({\mathbb R}^n, dx)}_{\mathbb C}.  $$
Its $L^{2}(\mu)$-norm then is given by
\begin{equation*}
\|F\|^2_{L^2(\mu)} = \sum_{n=0}^\infty n!\,|F^{(n)}|^2 < \infty.
\end{equation*}
Here $\mathcal{S}^{\prime} \ni x \mapsto :x^{\otimes n}\!: \, \in \widehat{\mathcal{S}^{\prime}({\mathbb R}^n)}$ denotes the Wick ordering, $\langle \cdot, \cdot \rangle$ the dual pairing between $\mathcal{S}^{\prime}({\mathbb R}^n)_{\mathbb C}$ and $\mathcal{S}^{\prime}({\mathbb R}^n)_{\mathbb C}$, extended in the first variable, in the sense of a square integrable function, to kernels $F^{(n)} \in \widehat{L^2({\mathbb R}^n,dx)}_{\mathbb C}$, see e.g.~\cite{Ob94}. The random variable $\big\langle F^{(n)}, :\cdot^{\otimes n}: \big\rangle$ coincides with the n-multiple Wiener integral of $F^{(n)} \in \widehat{L^2({\mathbb R}^n, dx)}_{\mathbb C}$, see e.g.~\cite{HKPS93}. For $F$ with finite chaos expansion we can define the number operator $N$ by 
$$L^2(\mu) \ni F = \sum_{n=0}^m \big\langle F^{(n)}, :\cdot^{\otimes n}: \big\rangle \mapsto NF := \sum_{n=0}^m n\big\langle F^{(n)}, :\cdot^{\otimes n}: \big\rangle \in L^2(\mu).$$

Following the definition of \cite{GNM22}, we define the test and regular generalized functions via their chaos decomposition
	\begin{multline*}
	{\mathcal G}_s := \Bigg{\{} \Phi = \sum_{n=0}^\infty \big\langle \Phi^{(n)} :\cdot^{\otimes n}: \big\rangle :  \Phi^{(n)} \in \widehat{L^2({\mathbb R}^n, dx)}_{\mathbb C}, \sum_{n=0}^\infty 2^{2sn} n! \,|\Phi^{(n)}|^2 < \infty \Bigg{\}}, \quad s \in {\mathbb R},
	\end{multline*}
where for $s>0$, we obtain a test function space and for $s<0$ a space of generalized functions with square integrable kernels. Equipped with
the norm $$\|\cdot\|_s := \|2^{sN} \cdot \|_{L^2(\mu)}$$ the space ${\mathcal G}_s, s \in {\mathbb R},$ is a Hilbert space. By taking a projective and inductive limit, respectively, we obtain the so-called Potthoff--Timpel triple: 
	\begin{eqnarray*}
	{\mathcal G} \subset L^2(\mu) \subset {\mathcal G}',
	\end{eqnarray*}
where 
	$$ {\mathcal G} := \bigcap_{q \in {\mathbb N}} {\mathcal G}_q \quad \text{and}\quad  {\mathcal G}^\prime := \bigcup_{q \in {\mathbb N}} {\mathcal G}_{-q}. $$
We have the following chain of spaces that can serve to define regularity of objects
	\begin{eqnarray*}
		({\mathcal S}) \subset {\mathcal G} \subset {\mathcal G}_r \subset {\mathcal G}_s \subset  L^2(\mu) \subset {\mathcal G}_{-s} \subset {\mathcal G}_{-r} \subset {\mathcal G}' \subset ({\mathcal S})', \quad r > s > 0.
	\end{eqnarray*}
Here $({\mathcal S})$ and $({\mathcal S})'$ denotes the space of Hida test functions and distributions, respectively. The characterization of the Potthof--Timpel triple has recently been improved in \cite{GNM22} via the Bargmann--Segal space combined with the concept of U-functionals. For this we introduce a Gaussian measure $\nu$ on ${\mathcal S}'_{\mathbb C}$ via 
	\begin{equation*}
		\int_{{\mathcal S}^\prime_{\mathbb C}} \exp\big(i\,\Re{\big\langle h, \overline{u} \big\rangle}\big) \,d\nu(u)
		= \exp\big(-\textstyle{\frac{1}{4}}\big\langle h, \overline{h} \big\rangle\big), \quad h \in {\mathcal S}_{{\mathbb C}}. 
	\end{equation*}
    I.e., $\nu$ is the product measure of the Gaussian measure on real and imaginary component of ${\mathcal S}^\prime_{\mathbb C}$ with covariance $(\cdot, \cdot)/2$, respectively. The Bargmann--Segal space is then given by
$$
		E^2(\nu) := \bigg\{H = \sum^\infty_{n=0} \langle H^{(n)}, \cdot^{\otimes n}\rangle : 
		H^{(n)}  \in \widehat{L^2({\mathbb R}^n, dx)}_{\mathbb C}, n \in {\mathbb N},
		\|H\|_{L^2(\nu)} < \infty \bigg\} \subset L^2(\nu).   
$$
As a key feature we have instead of the orthogonality of Wick ordered polynomials the orthogonality of monomials, i.e., 
	\begin{multline}\label{eqortho}
	    \left(\langle F^{(n)}, \cdot^{\otimes n} \rangle, \langle G^{(m)}, \cdot^{\otimes m}\rangle\right)_{L^2(\nu)} \\
		= \int_{{\mathcal S}^\prime_{\mathbb C}} \langle F^{(n)}, u^{\otimes n} \rangle
		\overline{\langle G^{(m)}, u^{\otimes m} \rangle} \,d\nu(u) = n!\,\langle F^{(n)}, \overline{G^{(n)}}\rangle \,\delta_{nm},
    \end{multline} 
where $F^{(n)} \in \widehat{L^2({\mathbb R}^n, dx)}_{\mathbb C}$, $G^{(m)} \in \widehat{L^2({\mathbb R}^m, dx)}_{\mathbb C}$,
	$n, m \in {\mathbb N}$, respectively.

The characterization of spaces from white noise analysis is done via the localization of the Laplace image in a suitable space of holomorphic functions. The role of the Laplace transform in Gaussian analysis is played by the $S$-transform:

\begin{defn}
Let $h \in {\mathcal S}_{\mathbb C}$. We denote the Wick exponential by $$:\exp(\langle h, \cdot \rangle): \, := \exp\left(\langle h, \cdot \rangle - \frac{1}{2}\langle h, h \rangle\right).$$
For an $F\in L^2(\mu)$ we define the $S$-transform by
$$SF(h) := \int_{{\mathcal S}'({\mathbb R})} :\exp(\langle h, \omega \rangle): \, F(\omega) \,d\mu(\omega),
				\quad h \in {\mathcal S}_{\mathbb C}.$$
Since for $h \in {\mathcal S}_{\mathbb C}$, the Wick exponential is a Hida test function,
we can define 
even for $\Phi \in (S)'$:
	\begin{align*}
				S\Phi(h) = \langle \! \langle :\exp(\langle h, \cdot \rangle):, \Phi \rangle \! \rangle,
				\quad h \in {\mathcal S}_{\mathbb C}.
			\end{align*}
\end{defn}
The $S$-transform of a Hida distributions can be identified with a U-functional:
\begin{defn}
\begin{enumerate}

    \item[(i)] A mapping $U: {\mathcal S}_{\mathbb C} \to {\mathbb C}$ is called a U-functional, iff
			\begin{align*}
				&(a) \,\,{\mathbb C} \ni z \mapsto U(f + zg) \in {\mathbb C} \,\text{ is entire
					for all } \, f, g \in {\mathcal S}; \\
				&(b) \,\mbox{there exist } 0 \le A, B < \infty \mbox{ and a continuous norm } \| \cdot \| \mbox{ on }
				{\mathcal S} \mbox{ such that} \\
				& \quad \quad |U(zf)| \le A\exp\big(B|z|^2\|f\|^2\big) \mbox{ for all } f \in {\mathcal S}, \, z \in {\mathbb C}.
			\end{align*}
            The set of U-functionals is denoted by $\mathcal{U}$.
    \item[(ii)] The set of finite dimensional orthogonal projection $\mathbb{P}$ is defined by
    \begin{multline*}
        \mathbb{P} := \bigg\{P:{\mathcal S}'_{\mathbb C} \to {\mathcal S}_{\mathbb C} : \\ P = \sum_{m=1}^n \langle e_m, \cdot \rangle e_m,
        e_m \in {\mathcal S}, 1 \le m \le n, \,\,\mathrm{orthonormal \,\, w.r.t.} (\cdot, \cdot), n \in \mathbb{N} \bigg\}.
    \end{multline*}
    
    \end{enumerate}
    \end{defn}

We obtain the following characterization of spaces:

\begin{thm}[\cite{PS91}, \cite{GNM22}]\label{thm:charaG}
Let $\Phi : (S) \to \mathbb{C}$ be linear. The characterization of spaces is as follows:\\
	\begin{tabular}{l l l}
		 $\Phi \in (S)'$ &$\Leftrightarrow$&  $S\Phi \in \mathcal{U}$; \\
	$\Phi \in \mathcal{G}$ &$\Leftrightarrow$& $S\Phi \in \mathcal{U}$ and $ 	\sup\limits_{P \in {\mathbb P}} \int_{{\mathcal S}_{{\mathbb C}}^\prime} \big|S\Phi(\lambda Pu)\big|^2 \,d\nu(u) < \infty
		\quad \mbox{for all } \lambda > 0$; \\
		$\Phi \in \mathcal{G}'$ &$\Leftrightarrow$& $S\Phi \in \mathcal{U}$ and $\sup\limits_{P \in {\mathbb P}} \int_{{\mathcal S}_{{\mathbb C}}^\prime} \big|S\Phi(\varepsilon Pu)\big|^2 \,d\nu(u) < \infty
		\quad \mbox{for some } \varepsilon > 0$; \\
		$\Phi \in L^2(\mu)$  &$\Leftrightarrow$& $S\Phi \in \mathcal{U}$ and $\sup\limits_{P \in {\mathbb P}} \int_{{\mathcal S}_{{\mathbb C}}^\prime} \big|S\Phi(Pu)\big|^2 \,d\nu(u) < \infty$.
	\end{tabular}
\end{thm}

In case we already know that $\Phi \in \mathcal{G}'$ we can omit taking the supremum in $P \in {\mathbb P}$ in the integrals above, since then the kernels of its chaos decomposition are already in $\widehat{L^2({\mathbb R}^n; dx)}_{\mathbb C}, n \in \mathbb{N}$.

\begin{lem}\label{le1}
	Let $\Phi \in \mathcal{G}'$, then there exist $\varepsilon > 0$ such that $S\Phi(\varepsilon \cdot) \in E^2(\nu)$ and 
    \begin{enumerate}
    \item[(i)]		
        \begin{equation*}
			\sup_{P \in {\mathbb P}} \int_{{\mathcal S}_{{\mathbb C}}^\prime} \big|S\Phi(\lambda Pu)\big|^2 \,d\nu(u)
            = \sum_{n=0}^{\infty} n! \lambda^{2n} |\Phi^{(n)}|^2
            = \int_{{\mathcal S}_{{\mathbb C}}^\prime} \big|S\Phi(\lambda u)\big|^2 \,d\nu(u)			\quad \mbox{for all } \lambda > 0. 
		\end{equation*}
In particular, the above expressions are finite for all $0<\lambda\le \varepsilon$. In general they might be infinite.
    \item[(ii)]
    \begin{equation*}
		\|\Phi\|^2_{s} = \int_{{\mathcal S}_{{\mathbb C}}^\prime}  \big|S\Phi(2^{s} u)\big|^2 \,d\nu(u)	\quad \mbox{for all } s \in {\mathbb R},
		\end{equation*}
       again with the convention that both sides may be infinite.
\end{enumerate}
\end{lem}

\begin{proof}
Let $\Phi\in\mathcal G'$. Then there exists $q\in\mathbb N$ such that $\Phi\in\mathcal G_{-q}$ and hence $\Phi$ admits a chaos expansion
\[
\Phi=\sum_{n=0}^\infty \langle \Phi^{(n)},:\cdot^{\otimes n}:\rangle,
\qquad \Phi^{(n)}\in \widehat{L^2(\mathbb R^n,dx)}_{\mathbb C},
\]
with
\[
\|\Phi\|_{-q}^2=\sum_{n=0}^\infty 2^{-2qn}\,n!\,|\Phi^{(n)}|^2<\infty.
\]
The S-transform of $\Phi$ is given by
\[
S\Phi(f)=\sum_{n=0}^\infty \langle \Phi^{(n)},f^{\otimes n} \rangle,
\qquad f \in \mathcal{S}_{\mathbb C},
\]
see \cite{GNM22}.

(i):
For any $\lambda > 0$ we have 
\begin{multline}\label{eqwithoutP}
		\sup_{P \in {\mathbb P}} \int_{{\mathcal S}_{{\mathbb C}}^\prime} \big|S\Phi(\lambda Pu)\big|^2 \,d\nu(u) = 	\sup_{P \in {\mathbb P}} \int_{{\mathcal S}_{{\mathbb C}}^\prime} S\Phi(\lambda Pu) \overline{S\Phi(\lambda Pu) }  \,d\nu(u)\\ 
		=\sup_{P \in {\mathbb P}} \int_{{\mathcal S}_{{\mathbb C}}^\prime} \sum_{n=0}^{\infty} \langle \Phi^{(n)}, \lambda^n (Pu)^{\otimes n}) \rangle \overline{\sum_{n=0}^{\infty} \langle \Phi^{(n)}, \lambda^n (Pu)^{\otimes n}) \rangle}  \,d\nu(u)\\ 
		=  \sup_{P \in {\mathbb P}} \sum_{n=0}^{\infty} n! \lambda^{2n} |P^{\otimes n}\Phi^{(n)}|^2
        = \sum_{n=0}^{\infty} n! \lambda^{2n} |\Phi^{(n)}|^2 \\
        =\int_{{\mathcal S}_{{\mathbb C}}^\prime} \sum_{n=0}^{\infty} \langle \Phi^{(n)}, \lambda^n u^{\otimes n} \rangle \overline{\sum_{n=0}^{\infty} \langle \Phi^{(n)}, \lambda^n u^{\otimes n}) \rangle}  \,d\nu(u)\\
        = \int_{{\mathcal S}_{{\mathbb C}}^\prime} \big|S\Phi(\lambda u)\big|^2 \,d\nu(u) = \|S\Phi(\lambda \cdot)\|^2_{L^2(\nu)},
\end{multline}
        because $\Phi^{(n)} \in \widehat{L^2({\mathbb R}^n; dx)}_{\mathbb C}, n \in \mathbb{N}$. The third equality follows by \eqref{eqortho} and the upper estimate in the fourth equality follows by taking a sequence of orthogonal projections corresponding to an orthonormal basis of $L^2(dx)$. Note that $S\Phi(\lambda \cdot): {\mathcal S}_{{\mathbb C}}^\prime \to \mathbb{C}$ is a well-defined function due to the fact $\Phi^{(n)} \in \widehat{L^2({\mathbb R}^n; dx)}_{\mathbb C}, n \in \mathbb{N}$, compare with \eqref{eqortho}, in the case the terms in \eqref{eqwithoutP}
        are finite. Now, due to Theorem \ref{thm:charaG} there exist
        $\varepsilon > 0$ such that $S\Phi(\varepsilon \cdot) \in E^2(\nu)$.
        
(ii): Follows directely from the definition of the norm in $\mathcal{G}_s$ and \eqref{eqwithoutP}.
\end{proof}


Let $(\mathcal S' ,\mathcal C_\sigma,\mu)$ be the white noise probability space.
For $F\in L^2(\mu)$ we write its Wiener--It\^{o} chaos expansion as
\[
F=\sum_{n=0}^\infty \big\langle F^{(n)},:\cdot^{\otimes n}:\big\rangle,
\qquad
F^{(n)}\in \widehat{L^2(\R^n,dx)}_{\C},
\]
so that
\[
\|F\|_{L^2(\mu)}^2=\sum_{n=0}^\infty n!\,|F^{(n)}|^2.
\]
The (closure of the) number operator $N$ on $L^2(\mu)$ is defined by
\[
NF:=\sum_{n=0}^\infty n \big\langle F^{(n)},:\cdot^{\otimes n}:\big\rangle,
\]
with domain
\[
\mathrm{Dom}(N):=\left\{F\in L^2(\mu):\sum_{n=0}^\infty n^2\,n!\,|F^{(n)}|^2<\infty\right\}.
\]
More generally, for $\alpha\in\R$, $\Phi \in \mathcal{G}'$, we define the fractional powers $(I+N^{\alpha/2})$ by
\[
(I+N^{\alpha/2})\Phi
:=\sum_{n=0}^\infty (1+n^{\alpha/2})\big\langle \Phi^{(n)},:\cdot^{\otimes n}:\big\rangle,
\]
whenever the right-hand side converges in $L^2(\mu)$.

\begin{defn}[Fractional Malliavin--Watanabe--Sobolev spaces]
Let $\alpha\in\R$. The space $\mathcal D^{\alpha,2}$ is defined by
\[
\mathcal D^{\alpha,2}
:=\left\{\Phi\in \mathcal{G}': (I+N^{\alpha/2})\Phi\in L^2(\mu)\right\},
\]
equipped with the Hilbert norm
\[
\|\Phi\|_{\alpha,2}
:=\big\|(I+N^{\alpha/2})\Phi\big\|_{L^2(\mu)}.
\]
Equivalently, in terms of chaos kernels,
\begin{equation}\label{eq:Dalpha2_chaos_norm}
\|\Phi\|_{\alpha,2}^2
=\sum_{n=0}^\infty (1+n^{\alpha})\,n!\,|\Phi^{(n)}|^2,
\end{equation}
and
\[
\Phi\in\mathcal D^{\alpha,2}
\quad\Longleftrightarrow\quad
\sum_{n=0}^\infty (1+n^{\alpha})\,n!\,|\Phi^{(n)}|^2<\infty.
\]
\end{defn}

\begin{rmk}
For $\alpha=m\in\N$, the norm \eqref{eq:Dalpha2_chaos_norm} is equivalent to the
classical Malliavin--Sobolev norm $\|F\|_{m,2}$ defined via Malliavin derivatives,
i.e.\ $F\in\mathcal D^{m,2}$ if and only if $F$ is $m$ times Malliavin differentiable
and $D^kF\in L^2(\mu;\,L^2(\R)^{\otimes k})$ for all $0\le k\le m$.
In the present work we use \eqref{eq:Dalpha2_chaos_norm} as the definition for all
$\alpha\in\R$.
\end{rmk}

The characterization of Malliavin--Watanabe regularity through the Potthoff--Timpel scale
$\{\mathcal G_s\}_{s\in\mathbb R}$ is somewhat unsatisfactory from the viewpoint of the
Malliavin--Watanabe--Sobolev scale $\{\mathcal D^{\alpha,2}\}_{\alpha\in\mathbb R}$.
The reason is that the $\mathcal G_s$-norms involve \emph{exponential} weights in the chaos order,
whereas $\mathcal D^{\alpha,2}$ is governed by \emph{polynomial} weights.
Consequently, any positive $\mathcal G_s$-regularity already forces arbitrarily high Malliavin
differentiability (and even $L^p$-bounds for Malliavin derivatives), see the next corollary
whose argument goes back to \cite{PT94}.

\begin{cor}
	Let $F \in L^2(\mu)$.
\begin{enumerate}
    \item[(i)] For all $m \in {\mathbb N}, p \ge 1$, there exists $\lambda \ge 1$ such that
	\begin{equation*}
		\int_{{\mathcal S}_{{\mathbb C}}^\prime} \big|SF(\lambda u)\big|^2 \,d\nu(u) < \infty
	\end{equation*}
    implies $F \in {\mathcal D}^{m,p}$. I.e., $F$ is $m$-times Malliavin differentiable
	and the Malliavin derivatives up to order $m$ are contained in $L^p(\mu)$.    
    \item[(ii)] In particular, if
    \begin{equation*}
			\int_{{\mathcal S}_{{\mathbb C}}^\prime} \big|SF(\lambda u)\big|^2 \,d\nu(u) < \infty
			\quad \mbox{for some } \lambda > 1,
		\end{equation*}
	then $F$ is infinitely often Malliavin differentiable
	and the Malliavin derivatives of arbitrary order are contained in $L^p(\mu)$, $1 \le p \le 2$.
\end{enumerate}
    \end{cor}
    
\begin{proof} (i):
First recall that the functions in $L^2(\mu)$ with finite order chaos decomposition are dense in all spaces of Malliavin differentiable functions and that one can represent their norms via the number operator $N$. Hence for large enough $t \ge 0$ and a positive finite constant $C_{t,m}$ it holds:
	\begin{multline}\label{eqhypo}
	\| F \|_{m,p} = \|(N+1)^\frac{m}{2}F \|_{L^p(\mu)}
    = \| e^{-tN}(N+1)^\frac{m}{2} e^{tN}F \|_{L^p(\mu)}    \\
    \le \| e^{-tN/2}(N+1)^\frac{m}{2}  e^{tN} F \|_{L^2(\mu)}
	\le C_{t,m} \| e^{tN} F \|_{L^2(\mu)} = C_{t,m} \|F \|_t 
	\end{multline}
    for all functions $F \in L^2(\mu)$ with finite order chaos decomposition. Here we used hypercontractivity of the semigroup generated by $-N$. Now choose $\lambda = e^t$ and combine with Lemma \ref{le1}(ii).
    
    (ii): Note that in the case $1 \le p \le 2$ we do not need hypercontractivity and the estimate as in \eqref{eqhypo} can be obtained for all $t > 0$.
\end{proof}

For any function $F \in L^2(\mu)$, $0 < \lambda < 1$, we have by Lemma \ref{le1}(i):
\begin{eqnarray*}
		\int_{{\mathcal S}_{{\mathbb C}}^\prime} \big|SF(\lambda u)\big|^2 \,d\nu(u) = \sum_{n=0}^{\infty} n! \lambda^{2n} |F^{(n)}|^2.
\end{eqnarray*}
Taking the derivative with respect to $\lambda$ yields:
$$
\partial_{\lambda} \int_{{\mathcal S}_{{\mathbb C}}^\prime} \big|SF(\lambda u)\big|^2 \,d\nu(u) =  \sum_{n=1}^{\infty} n! 2n \lambda^{2n-1} |F^{(n)}|^2.
$$
Taking the supremum over all $0 < \lambda < 1$ we obtain
$$
\sup_{0 < \lambda < 1} \partial_{\lambda} \int_{{\mathcal S}_{{\mathbb C}}^\prime} \big|SF(\lambda u)\big|^2 \,d\nu(u) =  2\sum_{n=1}^{\infty} n! n  |F^{(n)}|^2.   
$$
This leads to the following theorem: 
\begin{thm}\label{thsobo}
 	Let $F \in L^2(\mu)$, $m \in \mathbb{N}$.
 	Then $$F \in {\mathcal D}^{m,2}$$
$$ 	\Leftrightarrow$$
 	\begin{equation*}
     \sup_{0 < \lambda < 1} \partial^m_\lambda \left(\int_{{\mathcal S}_{{\mathbb C}}^\prime} \big|SF(\lambda u)\big|^2 \,d\nu(u) \right) < \infty.
     \end{equation*}
\end{thm}

\noindent\textbf{Proof:}
We give the proof for $m=1$, the general case follows by induction.\\
Let $$\sup_{0 < \lambda < 1} \partial_\lambda \left(\int_{{\mathcal S}_{{\mathbb C}}^\prime} \big|SF(\lambda u)\big|^2 \,d\nu(u) \right) < \infty,$$ 
then we have  $$\infty>2\sum_{n=1}^{\infty} n! n  |F^{(n)}|^2  >\sum_{n=1}^{\infty} n! (n+1)  |F^{(n)}|^2.$$ Thus, $F \in {\mathcal D}^{1,2}$.\\[.3cm]
Now let $F\in {\mathcal D}^{1,2}$. Then
\begin{multline*}
    \sup_{0 < \lambda < 1}\partial_\lambda \left(\int_{{\mathcal S}_{{\mathbb C}}^\prime} \big|SF(\lambda u)\big|^2 \,d\nu(u) \right)
    = 2\sum_{n=1}^{\infty} n! n  |F^{(n)}|^2\\ \le 2\sum_{n=1}^{\infty} n! (n+1)  |F^{(n)}|^2 = 2\|F\|^2_{1,2} <\infty.
\end{multline*}
\hfill $\qed$

\begin{rmk}
    The spaces ${\mathcal D}^{m,2}$ are therefore very close to $L^2(\mu)$, indeed much more close than the space $\mathcal{G}$. This is due to the fact that for ${\mathcal D}^{m,2}$ a polynomial weight is needed for the characterization, while in every space topologizing the Potthoff-Timpel triple, the weights are of exponential type.
\end{rmk}

Using the Riemann-Liouville fractional derivative, we can even characterize the spaces ${\mathcal D}^{m+\alpha,2}$ for $\alpha\in (0,1)$, $m \in \mathbb{N}$.

\begin{thm}\label{thm:refined_Malliavin}
 	Let $F \in L^2(\mu)$, $\alpha \in (0,1)$, $m \in \mathbb{N}$. Then
 	$$F \in {\mathcal D}^{m+\alpha,2}$$ 
$$ 	\Leftrightarrow$$
 	\begin{equation*}
     \sup_{0 < \lambda < 1} D^\alpha_{0+} \partial^m_\lambda \left(\int_{{\mathcal S}_{{\mathbb C}}^\prime} \big|SF(\lambda u)\big|^2 \,d\nu(u) \right) < \infty.
     \end{equation*}
\end{thm}

\noindent\textbf{Proof:}
We give the proof for $m=0$, the general case follows by an induction argument using Theorem \ref{thsobo}. \\
Let $$\sup_{0 < \lambda < 1} D^\alpha_{0+} \left(\int_{{\mathcal S}_{{\mathbb C}}^\prime} \big|SF(\lambda u)\big|^2 \,d\nu(u) \right) < \infty.$$ 
Then by Lemma \ref{le1}(i) together with Lemma \ref{lefrac} we have:  $$\sum_{n=1}^{\infty} n! \frac{\Gamma(2n+1)}{\Gamma(2n+1-\alpha)}  |F^{(n)}|^2 <\infty.$$
Now Lemma \ref{lemappOGamma} combined with Lemma \ref{lemseriesconv} implies
$$\infty> \sum_{n=1}^{\infty} n! (2n)^{\alpha} |F^{(n)}|^2 > \sum_{n=1}^{\infty} n! n^{\alpha}  |F^{(n)}|^2> \sum_{n=1}^{\infty} n! |F^{(n)}|^2.$$ 
Hence, we obtain: $$\sum_{n=1}^{\infty} n! (1+n^{\alpha})  |F^{(n)}|^2<\infty.$$
Thus, $F \in {\mathcal D}^{\alpha,2}$.\\[.3cm]
Now let $F\in {\mathcal D}^{\alpha,2}$. Then as above we have
\begin{eqnarray}\label{eq8}
\sup_{0 < \lambda < 1} D^\alpha_{0+}\left(\int_{{\mathcal S}_{{\mathbb C}}^\prime} \big|SF(\lambda Pu)\big|^2 \,d\nu(u) \right)
    = \sum_{n=1}^{\infty} n! \frac{\Gamma(2n+1)}{\Gamma(2n+1-\alpha)}   |F^{(n)}|^2.
\end{eqnarray}
The series in \eqref{eq8} is by Lemma \ref{lemappOGamma} together with Lemma \ref{lemseriesconv} convergent if and only if 
$$\sum_{n=1}^{\infty} n! n^{\alpha}  |F^{(n)}|^2<\infty.$$
Since
$$\sum_{n=1}^{\infty} n! n^{\alpha}  |F^{(n)}|^2 \le \sum_{n=1}^{\infty} n! (1+n^{\alpha})  |F^{(n)}|^2=\|F\|_{{\mathcal D}^{\alpha,2}}<\infty,$$
the assertion is shown.
\hfill $\qed$

We can also characterize the dual spaces of the $m$-times Malliavin differentiable functions:
\begin{thm}\label{thsobominus}
 	Let $\Phi \in \mathcal{G}'$, $m \in \mathbb{N}$.
 	Then $$\Phi \in {\mathcal D}^{-m,2}$$
$$ 	\Leftrightarrow$$
 	\begin{equation*}
 		\int_0^1 \int_0^{\lambda_m} \cdots \int_0^{\lambda_2} \int_{{\mathcal S}_{{\mathbb C}}^\prime} \big|S\Phi(\lambda_1 u)\big|^2 \,d\nu(u) \, d\lambda_1 \cdots d\lambda_m < \infty.
 	\end{equation*}
\end{thm}
\noindent\textbf{Proof:}
We give the proof for $m=1$, the general case follows by induction.\\
Since $\Phi \in \mathcal{G}'$, by Lemma \ref{le1}(i) we have that
\begin{eqnarray*}
		[0,1] \ni \lambda \mapsto \int_{{\mathcal S}_{{\mathbb C}}^\prime} \big|S\Phi(\lambda u)\big|^2 \,d\nu(u) = \sum_{n=0}^{\infty} n! \lambda^{2n} |\Phi^{(n)}|^2 \in [0, \infty]
\end{eqnarray*}
is well-defined and monotonically increasing. 
Let $$\int_0^1 \int_{{\mathcal S}_{{\mathbb C}}^\prime} \big|S\Phi(\lambda u)\big|^2 \,d\nu(u) \,d\lambda < \infty,$$ 
then we have via Fubini 
\begin{multline*}
		\infty > \int_0^1 \sum_{n=0}^{\infty} n! \lambda^{2n} |\Phi^{(n)}|^2 \, d\lambda
        = \sum_{n=0}^{\infty} n! \int_0^1 \lambda^{2n} \,d\lambda \, |\Phi^{(n)}|^2 \\
        = \sum_{n=0}^{\infty} \frac{n!}{2n+1} |\Phi^{(n)}|^2
        \ge \frac{1}{2}\sum_{n=0}^{\infty} \frac{n!}{n+1} |\Phi^{(n)}|^2 
        = \frac{\|\Phi\|^2_{-1,2}}{2}.
\end{multline*}
Thus, $\Phi \in {\mathcal D}^{-1,2}$.\\[.3cm]
Now, let $\Phi \in {\mathcal D}^{-1,2}$. Then
\begin{eqnarray*}
\int_0^1 \int_{{\mathcal S}_{{\mathbb C}}^\prime} \big|S\Phi(\lambda Pu)\big|^2 \,d\nu(u) \,d\lambda
    = \sum_{n=0}^{\infty} \frac{n!}{2n+1} |\Phi^{(n)}|^2
        \le \sum_{n=0}^{\infty} \frac{n!}{n+1} |\Phi^{(n)}|^2 
        = \|\Phi\|^2_{-1,2} < \infty. 
\end{eqnarray*}
\hfill $\qed$

Again using the Riemann-Liouville fractional integral, we can refine the characterization of the dual spaces for any $\alpha\in (0,1)$. We have: 

\begin{thm}\label{fraccharD}
 	Let $\Phi \in \mathcal{G}'$, $\alpha \in (0,1)$, $m \in \mathbb{N}$.
 	Then $$\Phi \in {\mathcal D}^{-m-\alpha,2}$$
$$ 	\Leftrightarrow$$
 	\begin{equation*}
 		\frac{1}{\Gamma(\alpha)}\int_0^1 \frac{1}{(1-\lambda_{m+1})^{1-\alpha}} \int_0^{\lambda_{m+1}} \cdots \int_0^{\lambda_2} \int_{{\mathcal S}_{{\mathbb C}}^\prime} \big|S\Phi(\lambda_1 u)\big|^2 \,d\nu(u) \, d\lambda_1 \cdots d\lambda_{m+1}  < \infty.
 	\end{equation*}
\end{thm}
\noindent\textbf{Proof:}
We give the proof for $m=0$ and $\alpha\in(0,1)$, the general case follows by an induction argument together with Theorem \ref{thsobominus}.\\
Since $\Phi \in \mathcal{G}'$, by Lemma \ref{le1}(i) we have that
\begin{eqnarray*}
		[0,1] \ni \lambda \mapsto \int_{{\mathcal S}_{{\mathbb C}}^\prime} \big|S\Phi(\lambda u)\big|^2 \,d\nu(u) = \sum_{n=0}^{\infty} n! \lambda^{2n} |\Phi^{(n)}|^2 \in [0, \infty]
\end{eqnarray*}
is well-defined and monotonically increasing. 
Let $$\frac{1}{\Gamma(\alpha)}\int_0^1 \frac{1}{(1-\lambda)^{1-\alpha}} \int_{{\mathcal S}_{{\mathbb C}}^\prime} \big|S\Phi(\lambda u)\big|^2 \,d\nu(u) \,d\lambda < \infty,$$ 
then we have via Fubini 
\begin{multline*}
	\infty >\frac{1}{\Gamma(\alpha)}\int_0^1 \frac{1}{(1-\lambda)^{1-\alpha}} \int_{{\mathcal S}_{{\mathbb C}}^\prime} \big|S\Phi(\lambda u)\big|^2 \,d\nu(u) \,d\lambda \\
        = \sum_{n=0}^{\infty} n! \frac{1}{\Gamma(\alpha)}\int_0^1 \frac{1}{(1-\lambda)^{1-\alpha}}  \lambda^{2n}
       \,d\lambda \, |\Phi^{(n)}|^2
        =\sum_{n=0}^{\infty} n! \frac{\Gamma(2n+1)}{\Gamma(2n+1+\alpha)} |\Phi^{(n)}|^2,    
\end{multline*}
by Lemma \ref{lefrac}.	
This implies the convergence of the series
$$
\sum_{n=0}^{\infty} \frac{n!}{1 + n^{\alpha}} |\Phi^{(n)}|^2 
        = \|\Phi\|^2_{-\alpha,2},
$$
by Lemma \ref{lemappOGamma} and Lemma \ref{lemseriesconv}.
Thus, $\Phi \in {\mathcal D}^{-\alpha,2}$.\\[.3cm]
Now let $\Phi \in {\mathcal D}^{-\alpha,2}$. Then
\begin{eqnarray*}
\frac{1}{\Gamma(\alpha)}\int_0^1 \frac{1}{(1-\lambda)^{1-\alpha}} \int_{{\mathcal S}_{{\mathbb C}}^\prime} \big|S\Phi(\lambda Pu)\big|^2 \,d\nu(u) \,d\lambda
    = \sum_{n=0}^{\infty} \frac{\Gamma(2n+1)}{\Gamma(2n+1+\alpha)} |\Phi^{(n)}|^2,
\end{eqnarray*}
which by Lemma \ref{lemappOGamma} and Lemma \ref{lemseriesconv} is convergent if and only if
\begin{equation}\label{eq42}
\sum_{n=0}^{\infty} \frac{n!}{1 + n^{\alpha}} |\Phi^{(n)}|^2 <\infty.
\end{equation} 
But the series in \eqref{eq42} just sums up to $\|\Phi\|^2_{-\alpha,2}$, which is finite by assumption.
\hfill $\qed$





\section{Applications}
\begin{rmk}\label{rm11}
Let $d\in \mathbb{N}$.  Donsker's delta for a process $\langle \eta, \cdot \rangle:=(\langle \eta_1, \cdot_1 \rangle,\dots,\langle \eta, \cdot_d \rangle), \eta \in L^2 (\mathbb{R},\mathbb{R}^d,dx)$, $\eta_k \neq 0, 1 \le k \le d$,
is defined as a generalized random variable on $\mathcal{S}^{\prime}(\mathbb{R}, \mathbb{R}^d)$ given by the product of the Donsker's deltas of the component processes, see e.g. \cite{SW93}:
\begin{eqnarray*}
 \delta (\langle \eta_1, \cdot_1 \rangle,\dots,\langle \eta_d, \cdot_d \rangle) = \prod_{k=1}^d \delta(\langle \eta_k, \cdot_k \rangle).
\end{eqnarray*}
Its $S$-transform in $\xi\in \mathcal{S}(\mathbb{R}, \mathbb{R}^d)$ is then given by
\begin{eqnarray}\label{STDD}
\prod_{k=1}^d \frac{1}{\sqrt{(2\pi) \langle\eta_k, \eta_k\rangle}} \exp\bigg(-\frac{1}{2 \langle\eta_k, \eta_k\rangle} \langle \xi_k, \eta_k \rangle^2)\bigg).
\end{eqnarray}
Hence, by denoting $\nu_d$ the measure on ${\mathcal S}_{{\mathbb C}}^\prime(\mathbb{R}, \mathbb{R}^d)$ obtained by the $d$-times product of the measure $\nu$ on ${\mathcal S}_{{\mathbb C}}^\prime$ and $P_k$ the k-th component of a orthogonal projection $P:{\mathcal S}_{{\mathbb C}}^\prime(\mathbb{R}, \mathbb{R}^d) \to {\mathcal S}_{{\mathbb C}}(\mathbb{R}, \mathbb{R}^d)$ we obtain
\begin{eqnarray}
   && \sup_{P\in \mathbb{P}} \int_{\mathcal{S}^\prime_{{\mathbb C}}(\mathbb{R}, \mathbb{R}^d)} \Bigg|S\delta (\langle \eta_1, \cdot_1 \rangle,\dots,\langle \eta_d, \cdot_d \rangle)(\lambda Pu) \Bigg|^2 d\nu_d(u) \nonumber \\
   &=&\sup_{P\in \mathbb{P}}\int_{\mathcal{S}^\prime_{{\mathbb C}}(\mathbb{R}, \mathbb{R}^d)} \Bigg|\prod_{k=1}^d \frac{1}{\sqrt{(2\pi) \langle\eta_k, \eta_k\rangle}} \exp\bigg(-\frac{1}{2 \langle\eta_k, \eta_k\rangle} \langle \lambda P_k u_k, \eta_k \rangle^2)\bigg)\Bigg|^2\, d\nu_d(u) \nonumber \\
       &=&\sup_{P\in \mathbb{P}}\prod_{k=1}^d \int_{\mathcal{S}^\prime_{{\mathbb C}}(\mathbb{R})} \Bigg| \frac{1}{\sqrt{(2\pi) \langle\eta_k, \eta_k\rangle}} \exp\bigg(-\frac{1}{2 \langle\eta_k, \eta_k\rangle} \langle \lambda P_k u_k, \eta_k \rangle^2)\bigg)\Bigg|^2\, d\nu(u_k),\label{productdonsk}
\end{eqnarray}
where in Equality \eqref{productdonsk} we used the independence of the component processes. 
Now we see from \eqref{productdonsk} that the existence of a Donsker's delta as a generalized function in $\mathcal{G}'$ in the $d$-dimensional case can be reduced to showing existence of Donsker's delta in $\mathcal{G}'$ in the one $1$-dimensional setting, which is due to \cite{GKS97}.
\end{rmk}

\begin{thm}[Regularity of Donsker's delta]
For any centered $\mathbb{R}^d$-valued Gaussian process $(X_t)_{t \ge 0}$ having independent components  $X^{k}_t = \langle f^{k}_t, \cdot\rangle, k=1,\dots,n$, with $0 \neq f^{(k)}_t \in L^2(\mathbb{R}, dx)$, $t \ge 0$, we have that
 $$\delta(X_t) \in \mathcal{D}^{-m,2}\setminus \mathcal{D}^{-\frac{d}{2},2}, \text{ for all }\,  m>\frac{d}{2}.$$
\end{thm}

\begin{proof} We first consider the one dimensional case.
Our aim is to apply Theorem \ref{fraccharD}. We know from Remark \ref{rm11} that Donsker's delta is in $\mathcal{G}'$ and its explicit form of the $S$-transform, see \eqref{STDD}. Hence, we are left to show integrability of
\begin{multline*}
    [0,1] \ni \lambda \mapsto \int_{\mathcal{S}^\prime_{\mathbb{C}}} |S(\delta(\langle f_t, \cdot\rangle))(\lambda u) |^2 \, d\nu(u) \\
    = \frac{1}{2\pi \|f_t\|^2} \int_{\mathcal{S}^{\prime}_{\mathbb{C}}} \exp\left(-\frac{\lambda^2}{2} \left\langle u, \frac{f_t}{\|f_t\|} \right\rangle^2 -\frac{\lambda^2}{2} \left\langle \overline{u}, \frac{f_t}{\|f_t\|} \right\rangle^2\right) d\nu(u).
\end{multline*}
Writing $u=v+iw$, $v, w \in \mathcal{S}^\prime$, we obtain $\langle u, \frac{f_t}{\|f_t\|} \rangle= \langle v, \frac{f_t}{\|f_t\|} \rangle +i \langle w, \frac{f_t}{\|f_t\|} \rangle$. Hence, using the image measures of these random variables and independence of imaginary and real part we obtain:
\begin{eqnarray}
 &&\int_{\mathcal{S}^{\prime}_{\mathbb{C}}} \exp\left(-\frac{\lambda^2}{2} \left(\left\langle u, \frac{f_t}{\|f_t\|} \right\rangle^2+ \left\langle \overline{u}, \frac{f_t}{\|f_t\|} \right\rangle^2 \right)\right) \,d\nu(u) \notag\\
    &=& \frac{1}{\pi}\int_{\mathbb{R}^2} \exp(-\lambda^2(x^2-y^2)) \exp(- x^2 -y^2) \, dx \, dy \notag\\
      &=&   \frac{1}{\pi}\int_{\mathbb{R}^2} \exp(-x^2 (\lambda^2+1) -y^2(1-\lambda^2)) \, dx \, dy \notag\\
      &=& \frac{1}{\sqrt{\lambda^2+1}}\frac{1}{\sqrt{1-\lambda^2}}\label{eq43}.
\end{eqnarray}

\noindent Let us now consider an arbitrary $d\in \mathbb{N}$. Accordingly to Equation \eqref{productdonsk} and \eqref{eq43} we obtain
$$\|S(\delta(\langle X_t, \cdot\rangle))(\lambda \cdot)\|^2_{L^2(\nu)} =\prod_{k=1}^d \left(\frac{1}{2\pi \left\|f^{(k)}_t\right\|^2} \frac{1}{\sqrt{1+\lambda^{2}}}\frac{1}{\sqrt{1-\lambda^{2}}}\right).$$
Hence, to apply Theorem \ref{fraccharD}, we have to check integrability of
$$[0,1] \ni \lambda \mapsto \frac{1}{(\sqrt{1-\lambda^{2}})^d},$$
since the other terms do not have influence on integrability. They can be bounded from above and from below away from zero by suitable constants.

Let $d=2n$. In that case is
$$\frac{1}{(\sqrt{1-\lambda^{2}})^d} =\frac{1}{(1-\lambda^{2})^n}=\frac{1}{(1+\lambda)^n}\frac{1}{(1-\lambda)^n}.$$
For the finiteness of the integral just the second term is of interest again. We have
$$\int_0^x \frac{1}{(1-\lambda)^n} \, d\lambda = \int_{1-x}^1 \frac{1}{y^n} \, dy=  -\frac{1}{(n-1)(1-x)^{n-1}}+\frac{1}{n-1}, \quad 0 \le x < 1.$$
By $n$-times iterated integration, we obtain except for some finite additive constant and prefactor the primitive function $$\ln(1-x),$$
which tends to $-\infty$ for $x$ approaching 1.
Hence we need to apply another Riemann--Liouville fractional integral of order $\alpha$. Omitting a finite prefactor we have
$$ \int_0^1 \ln(1-\lambda)(1-\lambda)^{\alpha-1} \, d \lambda = \int_0^1 \ln(y)\,y^{\alpha-1}\,dy.$$

\noindent Integration by parts yields
\begin{align*}
\lim_{s\searrow 0}\int_s^1 \ln(y)y^{\alpha-1}dy 
&= \lim_{s\searrow 0} \frac{y^\alpha}{\alpha}\ln(y)\bigg|_{s}^{1}
-\lim_{s\searrow 0}\int_s^1 \frac{y^\alpha}{\alpha} \frac{1}{y}dy \\
&= \frac{1}{\alpha}\ln(1) - \lim_{s\searrow 0} \frac{s^\alpha}{\alpha}\ln(s) - \frac{1}{\alpha}\int_0^1 y^{\alpha-1}\,dy=-\frac{1}{\alpha^2},
\end{align*}
which is a finite value. Thus, for $d=2n$ we have $\delta(X_t) \in \mathcal{D}^{-\frac{d}{2}-\varepsilon,2}, \text{ for all } \varepsilon>0.$

For $d=2n+1$, we write: 
$$\frac{1}{(\sqrt{1-\lambda^{2}})^d} =\frac{1}{(1-\lambda^{2})^n}\frac{1}{\sqrt{1-\lambda^{2}}}=\frac{1}{(1-\lambda)^{n+\frac{1}{2}}}\frac{1}{(1+\lambda)^{n+\frac{1}{2}}}.$$
For the finiteness of the integral only the first term is of interest. We have
$$\int_0^x \frac{1}{(1-\lambda)^{n+\frac{1}{2}}} \, d\lambda = \int_{1-x}^1 \frac{1}{y^{n+\frac{1}{2}}} \, dy=  -\frac{1}{({n-\frac{1}{2}})(1-x)^{{n-\frac{1}{2}}}}+\frac{1}{{n-\frac{1}{2}}}, \quad 0 \le x < 1.$$
Hence, by $n$-times iterated integration and again omitting a finite additive constant and prefactor, we obtain the primitive function
$$ \frac{1}{\sqrt{1-x}},$$
which is divergent for $x$ to 1. Applying the Riemann--Liouville integral of order $\alpha$ we obtain
$$\int_0^1 \frac{(1-\lambda)^{\alpha -1}}{\sqrt{1-\lambda}} d\lambda = \int_0^1 (1-\lambda)^{\alpha -\frac{3}{2}} d\lambda,$$
which is finite if and only if $-1<\alpha -\frac{3}{2}$, i.e., $\frac{1}{2}<\alpha.$ Thus, for $d=2n+1$ we have $\delta(X_t) \in \mathcal{D}^{-\frac{d}{2}-\varepsilon,2},$ for all $\varepsilon>0$.\end{proof} 

After showing regularity of Donsker's delta, we now analyze self-intersection local times. The aim is to show Malliavin smoothness. Here we concentrate on the case $d=1$, for the sake of simplicity.

\begin{defn}\label{defn_SILT}
    The self-intersection local time of a centered Gaussian process $X_t=\langle f_t, \cdot \rangle$, $0 \neq f_t \in L^{2}(\mathbb{R}, dx)$, $0 \le t \le T$, is defined by
$$    \int_0^T \int_0^t \delta(X_t-X_s) \, dsdt$$
in the sense of a Bochner integral in a suitable Hilbert space being a subset of $(\mathcal{S})^\prime$.
\end{defn}
\begin{defn}\label{defn:fBm}
    The centered Gaussian process $B^H$ with covariance
    $$
    \mathbb{E}(B_t^H B_s^H)=\frac{1}{2}\bigg( |t|^{2H} + |s|^{2H} -|t-s|^{2H}\bigg),\quad t,s >0,
    $$
    is called a one-dimensional fractional Brownian motion with Hurst parameter $H\in (0,1)$.
\end{defn}
\begin{thm}\cite{HN05, Hu2001}
For $d=1$ the self-intersection local time of a fractional Brownian motion with Hurst parameter $H\in (0,1)$ is in $\mathcal{D}^{1,2}$ for all $H\in (0,1)$.
\end{thm}
Indeed, we find the following generalization.

\begin{thm}\label{thm:silt}
Consider the one-dimensional Gaussian process $X = (\langle f_t , \cdot \rangle)_{0 \le t \le T}$, with $f_t\in L^2(\mathbb{R}, dx), 0 \le t \le T$. We use the notation $f(t,s):=f_t-f_s, 0 \le s,t \le T$, and assume that $0 < \|f(t,s)\| \le M$ dtds-a.e on $[0,T]^2$ for some $M<\infty$ . Then the self-intersection local time of $X$
is an element in $\mathcal{D}^{1,2}$ if  
    \begin{equation}\label{intcon}
         \int_0^T \int_0^{t_1} \int_0^T \int_0^{t_2} \bigg(\|f(t_1,s_1)\|^2\|f(t_2,s_2)\|^2 -\big(\langle f(t_1,s_1), f(t_2,s_2) \rangle\big)^2\bigg)^{-\frac{3}{2}} \, ds_2 dt_2 ds_1 dt_1<\infty.
    \end{equation}  
\end{thm}

\begin{proof}
The proof follows in three steps of regularization. We first show that the self-intersection local time of $X$ is an element in $\mathcal{G}^{\prime}$ to eliminate the supremum over all finite dimensional projections via Lemma \ref{le1}. Then we show square integrability and finally Malliavin regularity.

We have for all $0 < \lambda < 1$, $P \in \mathbb{P}$: 
\begin{eqnarray}\label{eq_SILT}  &&\int_{\mathcal{S}^{\prime}_{\mathbb{C}}}\bigg|S\left(\int_0^T \int_0^t \delta(X_t-X_s) \, ds\, dt\right)(\lambda Pu)\,\bigg|^2 d\nu(u) \\
 &=&\int_0^T\! \!\int_0^{t_1} \! \! \int_0^T\! \!\int_0^{t_2} \! \!
 \int_{\mathcal{S}^{\prime}_{\mathbb{C}}} S(\delta(X_{t_1}-X_{s_1}))(\lambda Pu) \,\overline{S(\delta(X_{t_2}-X_{s_2}))(\lambda Pu)} \, d\nu(u) ds_2 dt_2 ds_1 dt_1. \nonumber
\end{eqnarray}
For the integrand in Equation \eqref{eq_SILT} we use the notation $f(t,s) = f_t-f_s$ as in the formulation of Theorem \ref{thm:silt} and by Lemma \ref{lem_multivar} we obtain: 
\begin{eqnarray}\label{eq:inter}
  &&  \int_{\mathcal{S}^{\prime}_{\mathbb{C}}} S(\delta(\langle f(t_1,s_1), \cdot\rangle))(\lambda Pu) \overline{S(\delta(\langle f(t_2,s_2), \cdot\rangle))(\lambda Pu)} d\nu(u) \nonumber \\
&&\quad= \frac{1}{2\pi \|f(t_1,s_1)\| \|f(t_2,s_2)\|} \frac{1}{\sqrt{1-\sigma_P^2\lambda^4} }, 
\end{eqnarray}
with $$\sigma_P= \frac{\langle Pf(t_1,s_1), Pf(t_2,s_2) \rangle}{\|f(t_1,s_1)\| \|f(t_2,s_2)\|}\in [-1,1] .$$

Therefore we obtain in the case $\lambda = \frac{1}{2}$:
\begin{eqnarray*}
 && \int_{\mathcal{S}^{\prime}_{\mathbb{C}}}\bigg|S\left(\int_0^T \int_0^t \delta(X_t-X_s) \, ds\, dt\right)\bigg(\tfrac{1}{2}Pu\bigg)\,\bigg|^2 d\nu(u) \\
&& = \int_0^T\! \!\int_0^{t_1} \! \!\int_0^T\! \!\int_0^{t_2}  \frac{1}{2\pi \|f(t_1,s_1)\| \|f(t_2,s_2)\|} \frac{1}{\sqrt{1-\frac{1}{16}\sigma_P^2} } \, ds_2 dt_2 ds_1 dt_1 \\
  && = \int_0^T\! \!\int_0^{t_1}\! \! \int_0^T\! \!\int_0^{t_2} \frac{ds_2 dt_2 ds_1 dt_1}{2\pi \sqrt{\big(\|f(t_1,s_1)\|^2\|f(t_2,s_2)\|^2 - \frac{1}{16}\big(\langle Pf(t_1,s_1), Pf(t_2,s_2) \rangle\big)^2}} \\
&& \le \int_0^T\! \!\int_0^{t_1} \! \!\int_0^T\! \!\int_0^{t_2} \frac{1}{ \|f(t_1,s_1)\| \|f(t_2,s_2)\|} \, ds_2 dt_2 ds_1 dt_1.
\end{eqnarray*}
Since the latter integral is independent of $P \in \mathbb{P}$ and finite by assumption, we obtain that the self-intersection local time is in $\mathcal{G}^{\prime}$ by Theorem \ref{thm:charaG}.

Now, in the limit $\lambda \nearrow 1$ we obtain from Equation \eqref{eq:inter}
\begin{eqnarray*}
 && \int_{\mathcal{S}^{\prime}_{\mathbb{C}}}\bigg|S\left(\int_0^T \int_0^t \delta(X_t-X_s) \, ds\, dt\right)(u)\,\bigg|^2 d\nu(u) \\
  && = \int_0^T\! \!\int_0^{t_1}\! \! \int_0^T\! \!\int_0^{t_2} \! \!\frac{ds_2 dt_2 ds_1 dt_1}{2\pi \sqrt{\big(\|f(t_1,s_1)\|^2\|f(t_2,s_2)\|^2 - \big(\langle f(t_1,s_1), f(t_2,s_2) \rangle\big)^2}} < \infty
\end{eqnarray*}
by the integrability Condition \eqref{intcon}. Hence, we obtain square-integrability of the self-intersection local time by Theorem \ref{thm:charaG} together with Lemma \ref{le1}.

Now let us consider the function
\begin{eqnarray}
    (0,1) \ni \lambda \mapsto I(\lambda) := \frac{1}{\sqrt{\|f(t_1,s_1)\|^2 \|f(t_2,s_2)\|^2-\langle f(t_1,s_1), f(t_2,s_2) \rangle^2\lambda^4} }\label{eqnsigmageneral}
\end{eqnarray}
obtained from Equation \eqref{eq:inter}. Deriving the function in Equation \eqref{eqnsigmageneral} we obtain:
$$
I'(\lambda) = \frac{ 1}{\|f(t_1,s_1)\| \|f(t_2,s_2)\|} \frac{d}{d\lambda}\bigg( \frac{1}{\sqrt{1-\sigma_{Id}^2\lambda^4} }\bigg). 
$$
By Lemma \ref{derivative_silt} we have
$$ 
\frac{d}{d\lambda}\bigg( \frac{1}{\sqrt{1-\sigma_{Id}^2\lambda^4} }\bigg) =\frac{2\sigma_{Id} \lambda^3}{(1-\sigma_{Id}^2 \lambda^4)^{\frac{3}{2}}}.
$$
Hence, by Theorem \ref{thsobo} we have
$$
\int_0^T \int_0^t \delta(X_t-X_s) \, ds\, dt \in \mathcal{D}^{1,2}
$$
if and only if 
\begin{eqnarray}\label{eqSILT}
    \sup_{0<\lambda<1}\int_0^T \int_0^{t_1} \int_0^T \int_0^{t_2}  \frac{1}{2\pi \|f(t_1,s_1)\| \|f(t_2,s_2)\|}\frac{2\sigma_{Id}^2 \lambda^3}{(1-\sigma_{Id}^2 \lambda^4)^{\frac{3}{2}}} \, ds_2 dt_2 ds_1 dt_1  <\infty.
\end{eqnarray}
The integrand in \eqref{eqSILT} can be rewritten as
\begin{eqnarray}\label{eqSILT2}
   && \frac{ 1}{2\pi \|f(t_1,s_1)\| \|f(t_2,s_2)\|}\frac{2\sigma_{Id}^2 \lambda^3}{(1-\sigma _{Id}^2 \lambda^4)^{\frac{3}{2}}} = \frac{\lambda^3}{\pi} \frac{ \sigma_{Id}^2\|f(t_1,s_1)\|^2 \|f(t_2,s_2)\|^2}{\|f(t_1,s_1)\|^3 \|f(t_2,s_2)\|^3(1-\sigma_{Id}^2 \lambda^4)^{\frac{3}{2}}} \notag\\
    &=&  \frac{\lambda^3}{\pi} \frac{\langle f(t_1,s_1), f(t_2,s_2) \rangle^2}{\bigg(\|f(t_1,s_1)\|^2\|f(t_2,s_2)\|^2 -\big(\langle f(t_1,s_1), f(t_2,s_2) \rangle\big)^2\lambda^4\bigg)^{\frac{3}{2}}}.
\end{eqnarray}
Thus, the integral in \eqref{eqSILT} can be estimated by
\begin{eqnarray*}
    &&\sup_{0<\lambda<1}\int_0^T \! \!\int_0^{t_1} \! \!\int_0^T \! \!\int_0^{t_2} \frac{ 1}{2\pi \|f(t_1,s_1)\| \|f(t_2,s_2)\|}\frac{2\sigma_{Id}^2 \lambda^3}{(1-\sigma_{Id}^2 \lambda^4)^{\frac{3}{2}}} \, ds_2 dt_2 ds_1 dt_1\\
    && \le \int_0^T \! \!\int_0^{t_1} \! \!\int_0^T \! \!\int_0^{t_2} \frac{\langle f(t_1,s_1), f(t_2,s_2) \rangle^2 \, ds_2 dt_2 ds_1 dt_1}{\pi \bigg(\|f(t_1,s_1)\|^2\|f(t_2,s_2)\|^2 -\big(\langle f(t_1,s_1), f(t_2,s_2) \rangle\big)^2\bigg)^{\frac{3}{2}}} < \infty
\end{eqnarray*}
by the assumed integrability Condition \eqref{intcon}. Hence, the claim follows by Theorem \ref{thsobo}.
\end{proof}

\begin{rmk}
Note that the proof of Theorem~\ref{thm:silt} makes essential use of Lemma~\ref{le1} in a very specific way. In particular, it is crucial to first identify the object of interest as a regular generalized function in~$\mathcal{G}^{\prime}$. Once this identification has been established, the subsequent arguments can be carried out without taking a supremum over finite-dimensional projections. This point is important: if the finite-dimensional projections were to appear inside the inner product, ensuring square integrability would become a highly nontrivial issue. By showing that the object belongs to~$\mathcal{G}^{\prime}$, we guarantee both that the fundamental expressions required in later steps are well defined and that the assumptions of Lemma~\ref{le1} are satisfied. Consequently, this step is not a detour but a necessary and structural intermediate step in the argument.
\end{rmk}
\begin{ex}
For fractional Brownian motion with Hurst parameter $0<H<1$ we have:
\begin{eqnarray*}\|f(t,s)\|^2&=&\mathbb{E}((B_t^H-B_s^H)^2) \\
&=& \mathbb{E}(((B_t^H)^2 -2 B_t^H B_s^H + (B_s^H)^2)\\
&=& t^{2H} - t^{2H} -s^{2H} +(t-s)^{2H} + s^{2H} =(t-s)^{2H}
\end{eqnarray*}
and 
\begin{eqnarray*}
\langle f(t_1,s_1), f(t_2,s_2) \rangle &=& \mathbb{E}((B_{t_1}^H-B_{s_1}^H)(B_{t_2}^H-B_{s_2}^H)) \\
&=& \frac{1}{2}\bigg(t_1^{2H} +t_2^{2H} -(t_1-t_2)^{2H} - t_1^{2H} -s_2^{2H} +(t_1-s_2)^{2H}\\
&&
- s_1^{2H} -t_2^{2H} +(s_1-t_2)^{2H}+s_1^{2H} +s_2^{2H} -(s_1-s_2)^{2H}\bigg)\\ 
&=& \frac{1}{2}\bigg((t_1-s_2)^{2H}  +(s_1-t_2)^{2H} -(t_1-t_2)^{2H} -(s_1-s_2)^{2H}\bigg).
\end{eqnarray*}
Hence, the criterion
$$\int_0^T \int_0^{t_1} \int_0^T \int_0^{t_2} \frac{1}{\sqrt{\big(\|f(t_1,s_1)\|^2\|f(t_2,s_2)\|^2 -\big(\langle f(t_1,s_1), f(t_2,s_2) \rangle\big)^2}} \, ds_2 dt_2 ds_1 dt_1 < \infty$$
is exactly the criterion for square-integrability of the self-intersection local as provided in \cite{HN05}.

In \cite{Xie} the criterion for the Malliavin smoothness is shown for arbitrary Gaussian processes with stationary increments. 

Note that we do not need an assumption on stationary increments in Theorem \ref{thm:silt} but just need that $0<\|f(t,s)\|\leq M<\infty$ $dtds$-a.e. The class of Gaussian processes we can treat is thus larger. In this sense our result is a generalization of \cite{Xie}. For fractional Brownian motion, since it has stationary increments, both criteria coincide.

Indeed $$\int_0^T \int_0^{t_1} \int_0^T \int_0^{t_2}  \frac{1}{\bigg(\|f(t_1,s_1)\|^2\|f(t_2,s_2)\|^2 -\big(\langle f(t_1,s_1), f(t_2,s_2) \rangle\big)^2\bigg)^{\frac{3}{2}}} \, ds_2 dt_2 ds_1 dt_1 <\infty$$ is fulfilled for the fractional Brownian motion for all $0<H<1$. Hence the self-intersection local time is Malliavin differentiable.
\end{ex}

Finally we study a class of useful densities from Gaussian analysis, the so called Gauss kernels. 
\begin{defn}
    Denote by $L(L^2(\mathbb{R},dx))$ the set of all $A:L^2(\mathbb{R},dx) \to L^2(\mathbb{R},dx)$ being linear and continuous. Then we define the set of Gauss kernels 
    $$GK:=\left\{\Phi_A \in (\mathcal{S})^\prime : S(\Phi_A)(\xi) = \exp\left(-\frac12 \langle \xi, (Id-A) \xi \rangle\right), \, \xi \in \mathcal{S}_{\mathbb{C}}, \, A \in L(L^2(\mathbb{R},dx)) \right\}.$$
    Here and below we denote the complexification of $A$ again by $A$. Note that $S(\Phi_A), \Phi_A \in GK,$ is a $\mathcal{U}$ functional.
\end{defn}

\begin{rmk}
\begin{itemize}
    \item[(i)] The function $S\Phi_A$ is indeed a $\mathcal{U}$ functional since $\Phi_A \in (\mathcal{S})^{\prime}$. Generalizations of the class of Gauss kernels can for example be used to construct Feynman path integrals. Thus, they are an intensively studied class of functions, see e.g.~\cite{HKPS93, GS96, BG10}.
    \item[(ii)] For nice enough $A$, i.e., if $\det(A)>0$ exists and $A$ has an inverse on a suitable subspace of $L^2(\mathbb{R},dx)$, one has explicitly $$\Phi_A= \frac{\exp(-\frac12 \langle \cdot, (A^{-1}-Id) \cdot \rangle)}{\sqrt{\det(A)}}.$$
\end{itemize}
\end{rmk}

Under which conditions on $A$ do we have a Gauss kernel $\Phi_A \in \mathcal{D}^{m,2}$ and for which $m \in \mathbb{N}$?

\begin{thm}
Let $A \in L(L^2(\mathbb{R},dx))$ diagonalize along an orthonormal basis $\{e_n\}_{n \in \mathbb{N}}$ with eigenvalues in the open interval $(0,2)$.
\begin{enumerate}
    \item[(i)]  If $0<\det\!\bigl(2A-A^2\bigr)<\infty,$ then $\Phi_A\in L^2(\mu)$.
    \item[(ii)] If $\sup_{0 < \lambda < 1}\frac{\partial^{\alpha}}{\partial \lambda^{\alpha}}\det\!\bigl(Id - (\lambda^2(Id-A))^2\bigr)^{-\frac{1}{2}}< \infty,$ then $\Phi_A\in \mathcal{D}^{\alpha,2}$, where $\alpha >0$ and we define $\frac{\partial^{\alpha}}{\partial \lambda^{\alpha}} := D_{0+}^{\alpha-n} \frac{\partial^{n}}{\partial \lambda^{n}} $ for $\alpha >n \in \mathbb{N}$.
\end{enumerate}
\end{thm}
\begin{proof}
Let $\xi \in \mathcal{S}_{\mathbb{C}}$ and $0 < \lambda < 1$. Denote by $P_m$ the orthogonal projection generated by $\{e_n\}_{1 \le n \le m}, m \in \mathbb{N}$. Then

\begin{eqnarray}\label{GKconv}
\langle P_m\xi, (Id-A) P_m \xi \rangle= \sum_{n=1}^m \langle \xi, e_n\rangle^2 \kappa_n,
\end{eqnarray}
where $(\kappa_n)_{n \in \mathbb{N}}$ is the sequence of eigenvalues of $Id-A$ corresponding to the basis $\{e_n\}_{n \in \mathbb{N}}$. From \eqref{GKconv} togehter with Fubini's theorem we can conclude
\begin{eqnarray*}
  & & \|S(\Phi_A)(\lambda P_m \cdot)\|^2_{L^2(\nu)} \\ &=& \int_{\mathcal{S}_{\mathbb{C}}^{\prime}} \exp\left(-\frac12 \lambda^2\langle P_m u, (Id-A) P_m u \rangle\right)\exp\left(-\frac12 \lambda^2\langle P_m \overline{u}, (Id-A) P_m \overline{u} \rangle\right) \, d \nu (u)\\
  &=& \int_{\mathcal{S}_{\mathbb{C}}^{\prime}} \prod_{n=1}^m \exp\left(-\frac12 \lambda^2\langle v+iw, e_n\rangle^2 \kappa_n\right) \exp\left(-\lambda^2\frac12 \langle v-iw, e_n\rangle^2 \kappa_n\right) \, d\nu(v,w)\\
  &=& \int_{\mathcal{S}_{\mathbb{C}}^{\prime}} \prod_{n=1}^m \exp\left(-\lambda^2 (\langle v, e_n\rangle^2 - \langle w, e_n\rangle^2 ) \kappa_n \right) \, d\nu(v,w)\\  
  &=& \prod_{n=1}^m \frac{1}{\pi}\int_{\mathbb{R}^2} \exp\left(-\lambda^2 (x^2-y^2) \kappa_n \right) \exp\left(-x^2 -y^2\right) \,dxdy = \prod_{n=1}^m I_n,
\end{eqnarray*}
where
\[
I_n := \frac{1}{\pi}
\int_{\mathbb{R}^2}
\exp\!\bigl(-(1+\lambda^2\kappa_n)x^2 - (1-\lambda^2\kappa_n)y^2\bigr)\, dxdy, \quad n \in \mathbb{N}.
\]
By our assumptions is $(1\pm \lambda^2\kappa_n)>0$, hence
\[
\int_{\mathbb{R}} \exp(-(1+\lambda^2\kappa_n)x^2) \, dx
= \sqrt{\frac{\pi}{1+\lambda^2\kappa_n}}, 
\qquad
\int_{\mathbb{R}} \exp(-(1-\lambda^2\kappa_n)y^2) \, dy 
= \sqrt{\frac{\pi}{1-\lambda^2\kappa_n}} .
\]
Thus
\[
I_n 
= \frac{1}{\pi} \int_{\mathbb{R}^2}
\exp\!\bigl(-(1+\lambda^2\kappa_n)x^2 - (1-\lambda^2\kappa_n)y^2\bigr)\, dx dy
= \frac{1}{\sqrt{1-(\lambda^2\kappa_n)^2}}
\]
and
\begin{eqnarray*}
\det\!\bigl(Id - (\lambda^2(Id-A))^2\bigr)^{-1/2}
= \lim_{m \to \infty}\prod_{n=1}^m \frac{1}{\sqrt{1-(\lambda^2\kappa_n)^2}} 
= \lim_{m \to \infty} \|S(\Phi_A)(\lambda P_m \cdot)\|^2_{L^2(\nu)}.
\end{eqnarray*}
Since
$\{|S(\Phi_A)(\lambda P_m \cdot)|^2\}_{m \in \mathbb{N}}$
is uniformly integrable, together with the existence of
\begin{eqnarray*}
\lim_{m \to \infty} \prod_{n=1}^m \exp\left(\pm \lambda^2 (\langle v, e_n\rangle)^2 \kappa_n \right) \in [0, \infty], \quad v \in \mathcal{S}^{\prime},
\end{eqnarray*}
due to the assumptions made, it follows
$$\|S(\Phi_A)(\cdot)\|^2_{L^2(\nu)} = \lim_{m \to \infty} \|S(\Phi_A)(\lambda P_m \cdot)\|^2_{L^2(\nu)} = \det\!\bigl(Id - (\lambda^2(Id-A))^2\bigr)^{-1/2}.$$
Now using the characterization given in Theorem \ref{thm:refined_Malliavin}, we obtain the assertion. \end{proof}

\bibliographystyle{alpha}

\appendix
\renewcommand{\thesection}{A}
\section*{\appendixname \ Technical  Lemmata}
\begin{defn}
    The left sided Riemann--Liouville fractional integral and derivative, respectively, of order $\alpha \in (0,1)$ of an absolutely continuous function $f: [0, \infty) \to \mathbb{R}$ are defined by:
    \begin{enumerate}
        \item[(i)] $$
    I_{0+}^{\alpha}f(x) := \frac{1}{\Gamma(\alpha)}\int_{0}^x \frac{f(t)}{(x-t)^{1-\alpha}} \, dt, \qquad x \ge 0,
    $$         
        \item[(ii)] $$
    D_{0+}^{\alpha}f(x) := \frac{1}{\Gamma(1-\alpha)}\frac{d}{dx}\int_{0}^x \frac{f(t)}{(x-t)^{\alpha}} \, dt, \qquad x \ge 0.
    $$     
    \end{enumerate}
\end{defn}

\begin{lem}\label{lefrac}
For $n\in \mathbb{N}$ we  have: 
    \begin{enumerate}
        \item[(i)] $$I_{0+}^{\alpha}x^n = \frac{\Gamma(n+1)}{\Gamma(n+1+\alpha)}\, x^{n+\alpha}, \qquad x \ge 0,$$
        \item[(ii)] $$D_{0+}^{\alpha}x^n = \frac{\Gamma(n+1)}{\Gamma(n+1-\alpha)}\, x^{n-\alpha}, \qquad x \ge 0.$$
    \end{enumerate}

\end{lem}
\begin{proof}(i):
We want to compute
\[
I(x) = \int_{0}^x \frac{t^n}{(x-t)^{1-\alpha}} \, dt
= \int_0^x t^n (x-t)^{\alpha-1}\,dt.
\]

We use the substitution
\[
t = xu, \qquad dt = x\,du, \qquad u \in [0,1].
\]
Then
\[
t^n = (xu)^n = x^n u^n,
\]
and
\[
x - t = x - xu = x(1-u)
\quad\Rightarrow\quad
(x-t)^{\alpha-1} = \bigl(x(1-u)\bigr)^{\alpha-1}
= x^{\alpha-1}(1-u)^{\alpha-1}.
\]

Thus the integral becomes
\begin{align*}
I(x)
&= \int_0^x t^n (x-t)^{\alpha-1}\,dt \\
&= \int_0^1 \bigl(x^n u^n\bigr) \bigl(x^{\alpha-1}(1-u)^{\alpha-1}\bigr) x \,du \\
&= x^{n+\alpha} \int_0^1 u^n (1-u)^{\alpha-1}\,du.
\end{align*}

We recognize the remaining integral as the Beta function:
\[
B(p,q) = \int_0^1 u^{p-1}(1-u)^{q-1}\,du.
\]
With \(p = n+1\) and \(q = \alpha\), we obtain
\[
\int_0^1 u^n (1-u)^{\alpha-1}\,du
= B(n+1,\alpha).
\]

Hence
\[
I(x) = x^{n+\alpha} B(n+1,\alpha).
\]

Using the identity
\[
B(p,q) = \frac{\Gamma(p)\Gamma(q)}{\Gamma(p+q)},
\]
we get
\[
B(n+1,\alpha)
= \frac{\Gamma(n+1)\Gamma(\alpha)}{\Gamma(n+1+\alpha)}.
\]

\noindent Therefore, the final result is
\[
\frac{1}{\Gamma(\alpha)}\int_{0}^x \frac{t^n}{(x-t)^{1-\alpha}} \, dt
= x^{n+\alpha} \frac{\Gamma(n+1)}{\Gamma(n+1+\alpha)}.
\]

(ii): Next we compute
\[
\frac{1}{\Gamma(1-\alpha)}\frac{d}{dx}\int_{0}^x \frac{t^n}{(x-t)^{\alpha}} \, dt
= \frac{1}{\Gamma(1-\alpha)}\frac{d}{dx}\int_{0}^x t^n (x-t)^{-\alpha} \, dt.
\]

First, compute the integral
\[
R(x) := \int_{0}^x t^n (x-t)^{-\alpha} \, dt.
\]

We use the substitution
\[
t = xu, \qquad dt = x\,du, \qquad u \in [0,1].
\]
Then
\[
t^n = (xu)^n = x^n u^n,
\]
and
\[
x - t = x - xu = x(1-u)
\quad\text{ plugging in yields }\quad
(x-t)^{-\alpha} = \bigl(x(1-u)\bigr)^{-\alpha}
= x^{-\alpha}(1-u)^{-\alpha}.
\]

Hence
\begin{align*}
R(x)
&= \int_{0}^x t^n (x-t)^{-\alpha}\,dt \\
&= \int_0^1 \bigl(x^n u^n\bigr)\bigl(x^{-\alpha}(1-u)^{-\alpha}\bigr)x\,du \\
&= x^{n+1-\alpha}\int_0^1 u^n (1-u)^{-\alpha}\,du.
\end{align*}

We recognize the remaining integral as a Beta function:
\[
B(p,q) = \int_0^1 u^{p-1}(1-u)^{q-1}\,du.
\]
With \(p = n+1\) and \(q = 1-\alpha\), we obtain
\[
\int_0^1 u^n (1-u)^{-\alpha}\,du = B(n+1,1-\alpha).
\]

Thus
\[
R(x) = x^{n+1-\alpha} B(n+1,1-\alpha).
\]

Using
\[
B(p,q) = \frac{\Gamma(p)\Gamma(q)}{\Gamma(p+q)},
\]
we get
\[
B(n+1,1-\alpha)
= \frac{\Gamma(n+1)\Gamma(1-\alpha)}{\Gamma(n+2-\alpha)}.
\]

Hence
\[
R(x) = x^{n+1-\alpha}\,\frac{\Gamma(n+1)\Gamma(1-\alpha)}{\Gamma(n+2-\alpha)}.
\]

Now go back to the original expression:
\[
\frac{1}{\Gamma(1-\alpha)}\frac{d}{dx}\int_{0}^x \frac{t^n}{(x-t)^{\alpha}} \, dt
= \frac{1}{\Gamma(1-\alpha)}\frac{d}{dx} R(x).
\]

We compute:
\[
\frac{d}{dx} R(x)
= \frac{\Gamma(n+1)\Gamma(1-\alpha)}{\Gamma(n+2-\alpha)} (n+1-\alpha)x^{n-\alpha}.
\]

Multiplying by \(\frac{1}{\Gamma(1-\alpha)}\) gives
\[
\frac{1}{\Gamma(1-\alpha)}\frac{d}{dx} R(x)
= \frac{\Gamma(n+1)}{\Gamma(n+2-\alpha)} (n+1-\alpha)x^{n-\alpha}.
\]

Use the Gamma function identity
\[
\Gamma(n+2-\alpha) = (n+1-\alpha)\,\Gamma(n+1-\alpha)
\]
to simplify:
\[
\frac{(n+1-\alpha)}{\Gamma(n+2-\alpha)} = \frac{1}{\Gamma(n+1-\alpha)}.
\]

Therefore
\[
\frac{1}{\Gamma(1-\alpha)}\frac{d}{dx}\int_{0}^x \frac{t^n}{(x-t)^{\alpha}} \, dt
= \frac{\Gamma(n+1)}{\Gamma(n+1-\alpha)} x^{n-\alpha}.
\]
\hfill \end{proof}

\begin{lem}\label{lemappOGamma}
 For all $\alpha \in (0,1), n \in \mathbb{N},$ we have that:
 \begin{enumerate}
     \item[(i)] $$ \frac{\Gamma(n+1)}{\Gamma(n+1-\alpha)} \sim n^{\alpha},\quad n\to \infty,$$
     \item[(ii)] $$ \frac{\Gamma(n+1)}{\Gamma(n+1+\alpha)} \sim n^{-\alpha},\quad n\to \infty.$$
 \end{enumerate}
\end{lem}
\begin{proof}
\noindent Fix $\alpha \in (0,1)$. We use Stirling's asymptotic formula
\[
\Gamma(x+1) \sim \sqrt{2\pi x}\,\left(\frac{x}{e}\right)^x, \qquad x\to\infty,
\]
\noindent which means
\[
\lim_{x\to\infty}
\frac{\Gamma(x+1)}{\sqrt{2\pi x}\,(x/e)^x} = 1.
\]


(i): Apply Stirling's formula to $x = n$ and $x = n-\alpha$:
\[
\Gamma(n+1) \sim \sqrt{2\pi n}\,\Big(\frac{n}{e}\Big)^n,
\qquad
\Gamma(n+1-\alpha) \sim \sqrt{2\pi (n-\alpha)}\,\Big(\frac{n-\alpha}{e}\Big)^{n-\alpha}.
\]

\noindent Hence
\begin{align*}
\frac{\Gamma(n+1)}{\Gamma(n+1-\alpha)}
&\sim
\frac{\sqrt{2\pi n}\,\big(\frac{n}{e}\big)^n}{\sqrt{2\pi (n-\alpha)}\,\big(\frac{n-\alpha}{e}\big)^{n-\alpha}}\\[0.5em]
&=
\sqrt{\frac{n}{n-\alpha}}\,
\frac{n^n e^{-n}}{(n-\alpha)^{n-\alpha} e^{-(n-\alpha)}}\\[0.5em]
&=
\sqrt{\frac{n}{n-\alpha}}\,
\frac{n^n}{(n-\alpha)^{n-\alpha}}\,e^{-\alpha}.
\end{align*}

\noindent Rewrite
\[
\frac{n^n}{(n-\alpha)^{n-\alpha}}
=
\frac{n^n}{(n-\alpha)^n}\,(n-\alpha)^\alpha
=
\left(\frac{n}{n-\alpha}\right)^n (n-\alpha)^\alpha.
\]

\noindent Thus
\[
\frac{\Gamma(n+1)}{\Gamma(n+1-\alpha)}
\sim
\sqrt{\frac{n}{n-\alpha}}\,
\left(\frac{n}{n-\alpha}\right)^n e^{-\alpha}\,(n-\alpha)^\alpha.
\]

\noindent Note that
\[
\left(\frac{n}{n-\alpha}\right)^n
=
\left(\frac{1}{1-\alpha/n}\right)^n \longrightarrow e^{\alpha}
\quad\text{as } n\to\infty,
\]
\noindent and
\[
\sqrt{\frac{n}{n-\alpha}} \longrightarrow 1, \text{ as well as}
\quad (n-\alpha)^\alpha \sim n^\alpha.
\]

\noindent Combining these limits yields,
\[
\frac{\Gamma(n+1)}{\Gamma(n+1-\alpha)}
\sim
\sqrt{\frac{n}{n-\alpha}}\,
\left(\frac{n}{n-\alpha}\right)^n e^{-\alpha}\,(n-\alpha)^\alpha
\sim
n^\alpha.
\]

\noindent Consequently,
\[
\frac{\Gamma(n+1)}{\Gamma(n+1-\alpha)} \sim n^\alpha.
\]

\medskip

(ii): Next we apply Stirling's formula to $x = n$ and $x = n+\alpha$:
\[
\Gamma(n+1) \sim \sqrt{2\pi n}\,\Big(\frac{n}{e}\Big)^n,
\qquad
\Gamma(n+1+\alpha) \sim \sqrt{2\pi (n+\alpha)}\,\Big(\frac{n+\alpha}{e}\Big)^{n+\alpha}.
\]

\noindent Then
\begin{align*}
\frac{\Gamma(n+1)}{\Gamma(n+1+\alpha)}
&\sim
\frac{\sqrt{2\pi n}\,\big(\frac{n}{e}\big)^n}{\sqrt{2\pi (n+\alpha)}\,\big(\frac{n+\alpha}{e}\big)^{n+\alpha}}\\[0.5em]
&=
\sqrt{\frac{n}{n+\alpha}}\,
\frac{n^n e^{-n}}{(n+\alpha)^{n+\alpha} e^{-(n+\alpha)}}\\[0.5em]
&=
\sqrt{\frac{n}{n+\alpha}}\,
\frac{n^n}{(n+\alpha)^{n+\alpha}}\,e^{\alpha}.
\end{align*}

\noindent Rewrite
\[
\frac{n^n}{(n+\alpha)^{n+\alpha}}
=
\frac{n^n}{(n+\alpha)^n}\,(n+\alpha)^{-\alpha}
=
\left(\frac{n}{n+\alpha}\right)^n (n+\alpha)^{-\alpha}.
\]

\noindent Therefore
\[
\frac{\Gamma(n+1)}{\Gamma(n+1+\alpha)}
\sim
\sqrt{\frac{n}{n+\alpha}}\,
\left(\frac{n}{n+\alpha}\right)^n e^{\alpha}\,(n+\alpha)^{-\alpha}.
\]

\noindent Now
\[
\left(\frac{n}{n+\alpha}\right)^n
=
\left(\frac{1}{1+\alpha/n}\right)^n \longrightarrow e^{-\alpha}
\quad\text{as } n\to\infty,
\]
\noindent and
\[
\sqrt{\frac{n}{n+\alpha}} \longrightarrow 1, 
\text{as well as} \quad (n+\alpha)^{-\alpha} \sim n^{-\alpha}.
\]

\noindent So
\[
\frac{\Gamma(n+1)}{\Gamma(n+1+\alpha)}
\sim
\sqrt{\frac{n}{n+\alpha}}\,
\left(\frac{n}{n+\alpha}\right)^n e^{\alpha}\,(n+\alpha)^{-\alpha}
\sim
n^{-\alpha}.
\]

\noindent Hence summarizing
\[
\frac{\Gamma(n+1)}{\Gamma(n+1+\alpha)} \sim n^{-\alpha}.
\]

 \end{proof}

\begin{lem}\label{lemseriesconv}  
    Let $(a_n)_n$ be a sequence of non-negative and $(g_n)_n$, $(f_n)_n$ be sequences of positive real numbers, fulfilling $\lim\limits_{n\to \infty }\frac{f_n}{g_n}=C>0$. Then 
    $$\sum_{n=0}^{\infty} g_n \cdot a_n < \infty \Leftrightarrow \sum_{n=0}^{\infty} f_n \cdot a_n <\infty.$$
\end{lem}
\begin{proof}Since $\lim_{n\to \infty }\frac{f_n}{g_n}=C>0$ we have that for all $\varepsilon>0$ we find an $N \in \mathbb{N}$ such that for all $n >N$ we have that $C-\varepsilon < \frac{g_n}{f_n} < C+\varepsilon$. This implies directly that
$$
(C-\varepsilon)f_n < g_n < (C+\varepsilon)f_n,
$$
since the $f_n$ are positive. Because $a_n \ge 0$, $n \in \mathbb{N}$, we obtain

$$
(C-\varepsilon)\sum_{n=N}^{\infty} f_n \cdot a_n \le \sum_{n=N}^{\infty} g_n \cdot a_n \le (C+\varepsilon)\sum_{n=N}^{\infty} f_n \cdot a_n,
$$
which implies the assertion via the mayorant convergence criterion. \hfill \end{proof}
\begin{lem}\label{lem_multivar}
Let $0 \neq f, g \in L^2(\mathbb{R},dx)$, $P \in \mathbb{P}$ and $0 < \lambda < 1$, then 
\begin{multline*}
 \int_{\mathcal{S}^{\prime}_{\mathbb{C}}} S(\delta(\langle f, \cdot\rangle))(\lambda Pu) \overline{S(\delta(\langle g, \cdot\rangle))(\lambda Pu)} d\nu(u) \\
 = \frac{1}{2\pi \|f\| \|g\|} \frac{ 1}{\sqrt{1-\sigma_P^2\lambda^4} }, \quad \sigma_P := \left\langle \frac{Pf}{\|f\|}, \frac{Pg}{\|g\|} \right\rangle.
\end{multline*}
\end{lem}
\begin{proof}
We compute
\begin{eqnarray*}
&&\int_{\mathcal{S}^{\prime}_{\mathbb{C}}} S(\delta(\langle f, \cdot\rangle))(\lambda Pu) \overline{S(\delta(\langle g, \cdot\rangle))(\lambda Pu)} d\nu(u)\\
&=& \frac{1}{2\pi \|f\| \|g\|} \int_{\mathcal{S}^{\prime}_{\mathbb{C}}} \exp\left(-\frac{\lambda^2}{2} \left(\left\langle Pu, \frac{f}{\|f\|} \right\rangle^2 + \left\langle \overline{Pu} , \frac{g}{\|g\|} \right\rangle^2 \right) \right)\, d\nu(u)\\
&=&\frac{1}{2\pi \|f\| \|g\|} \int_{\mathcal{S}^{\prime}_{\mathbb{C}}} \exp\left(-\frac{\lambda^2}{2} \left(\left\langle u, \frac{Pf}{\|f\|} \right\rangle^2 + \left\langle \overline{u} , \frac{Pg}{\|g\|} \right\rangle^2 \right) \right)\, d\nu(u)\\
&=&\frac{1}{2\pi^3 \|f\| \|g\| \sqrt{\mathrm{det(\Sigma)}}}  \int_{\mathbb{R}^4} \exp\bigg(-\frac{\lambda^2}{2} \big((x+iy)^2 + (v-iw)^2 \big) \bigg) \\
& & \cdot \exp\left(-\left\langle \begin{pmatrix} x\\y\\v\\w \end{pmatrix} \Sigma^{-1} \begin{pmatrix} x\\y\\v\\w \end{pmatrix} \right\rangle \right) \,dx dy dv dw, 
\end{eqnarray*}
where $\Sigma$ is given by
$$
\Sigma= \begin{pmatrix}
    1& 0 & \sigma_P & 0 \\
    0& 1 & 0& \sigma_P \\
     \sigma_P & 0 & 1 & 0 \\
   0 & \sigma_P & 0 & 1 \\
\end{pmatrix},  
$$
with $\sigma_P= \left\langle \frac{Pf}{\|f\|}, \frac{Pg}{\|g\|} \right\rangle $. Hence, we have a fully commutative block matrix and
\[
\det(\Sigma)=(1-\sigma_P^2)^2.
\]
We then find for the inverse: 

$$
\begin{pmatrix}
    I&A\\
    A&I
\end{pmatrix}^{-1} = (I-A^2)^{-1}\times\begin{pmatrix}
    I&-A\\
    -A&I
\end{pmatrix},
$$
where $\times$ here denotes block-wise matrix multiplication.
Note that $$(I-A^2)^{-1} = \begin{pmatrix}
    \frac{1}{1-\sigma_P^2}&0\\
    0& \frac{1}{1-\sigma_P^2}
\end{pmatrix},$$
hence,
$$
\Sigma^{-1}= \frac{1}{1-\sigma_P^2} \begin{pmatrix}
    1& 0 & -\sigma_P & 0 \\
    0& 1 & 0& -\sigma_P \\
     -\sigma_P & 0 & 1 & 0 \\
   0 & -\sigma_P & 0 & 1 \\
\end{pmatrix}.
$$

In addition, we can rewrite
$$
\big((x+iy)^2 + (v-iw)^2 \big) = \big(x^2 +i2xy -y^2 + v^2-2ivw -w^2 \big)= \left\langle \begin{pmatrix} x\\y\\v\\w \end{pmatrix}, A \begin{pmatrix} x\\y\\v\\w \end{pmatrix}\right\rangle,
$$
where 
$$
A= \begin{pmatrix}
    1 & i & 0 & 0 \\
    i & -1 & 0 & 0\\
    0 & 0 & 1& -i\\
    0 & 0 & -i& -1\\
\end{pmatrix}.
$$
Altogether we have: 
\begin{multline*}
 \int_{\mathbb{R}^4} \exp\left(-\frac{\lambda^2}{2} \big((x+iy)^2 + (v-iw)^2 \big) \right) \exp\left(-\left\langle \begin{pmatrix} x\\y\\v\\w \end{pmatrix} \Sigma^{-1} \begin{pmatrix} x\\y\\v\\w \end{pmatrix} \right\rangle \right) \,dxdydvdw \\
= \int_{\mathbb{R}^4}
\exp\left(-\left\langle \begin{pmatrix} x\\y\\v\\w \end{pmatrix}, \left(\tfrac{\lambda^2}{2}\begin{pmatrix}
    1 & i & 0 & 0 \\
    i & -1 & 0 & 0\\
    0 & 0 & 1& -i\\
    0 & 0 & -i& -1\\
\end{pmatrix} \right.\right.\right. + 
\\
\left.\left.\left.\tfrac{1}{1-\sigma_P^2} \begin{pmatrix}
    1& 0 & -\sigma_P & 0 \\
    0& 1 & 0& -\sigma_P \\
     -\sigma_P & 0 & 1 & 0 \\
   0 & -\sigma_P & 0 & 1 \\
\end{pmatrix}\right) \begin{pmatrix} x\\y\\v\\w \end{pmatrix} \right\rangle \right) \,dxdydvdw\\
=(2\pi)^2\sqrt{\mathrm{det}\begin{pmatrix}
    \lambda^2 + \frac{2}{1-\sigma_P^2}& i\lambda^2 & -\frac{2\sigma_P}{1-\sigma_P^2} & 0 \\
    i\lambda^2 & -\lambda^2+\frac{2}{1-\sigma_P^2} & 0 & -\frac{2\sigma_P}{1-\sigma_P^2}\\
    -\frac{2\sigma_P}{1-\sigma_P^2} & 0 & \lambda^2+\frac{2}{1-\sigma_P^2}& -i\lambda^2\\
    0 & -\frac{2\sigma_P}{1-\sigma_P^2} & -i\lambda^2& -\lambda^2+\frac{2}{1-\sigma_P^2}\\
\end{pmatrix}^{-1}}=\pi^2 \frac{ 1-\sigma_P^2}{\sqrt{1-\sigma^2\lambda^4} }
\end{multline*}
Hence,
\begin{eqnarray*}
\int_{\mathcal{S}^{\prime}_{\mathbb{C}}} S(\delta(\langle f, \cdot\rangle))(\lambda u) \overline{S(\delta(\langle g, \cdot\rangle))(\lambda u)} d\nu(u)
&=& \frac{1}{2\pi \|f\| \|g\|(1-\sigma_P^2)} \frac{ 1-\sigma_P^2}{\sqrt{1-\sigma_P^2\lambda^4} }\\
&=&\frac{1}{2\pi \|f\| \|g\|} \frac{ 1}{\sqrt{1-\sigma_P^2\lambda^4} }.
\end{eqnarray*}
\end{proof}

\begin{lem}\label{derivative_silt}
We have: 
$$
\frac{d}{dx}\left( \frac{1}{\sqrt{a - b x^{4}}} \right) = \frac{2 b x^{3}}{\left( a - b x^{4} \right)^{3/2}},
$$

$$
\frac{d^{2}}{dx^{2}}\left( \frac{1}{\sqrt{a - b x^{4}}} \right)
= \frac{6 b x^{2}\,\bigl(a + b x^{4}\bigr)}{(a - b x^{4})^{5/2}},
$$

$$
\frac{d^{3}}{dx^{3}}\left( \frac{1}{\sqrt{a - b x^{4}}} \right)
= 
\frac{12a^{2}b\,x \;+\; 84ab^{2}x^{5} \;+\; 24 b^{3} x^{9}}
{(a - b x^{4})^{7/2}}.
$$
\end{lem}

\end{document}